\documentclass[11pt]{amsart}
\usepackage{mathrsfs}
\usepackage{amssymb}
\usepackage[T1]{fontenc}
\usepackage{amscd}
\usepackage{amsmath, amssymb}
\usepackage{amsthm}
\usepackage{amsfonts}
\usepackage[colorlinks,linkcolor=blue,citecolor=blue, pdfstartview=FitH]
{hyperref}
\usepackage{backref}
\usepackage{amscd}
\usepackage{easybmat}
\usepackage{mathrsfs}
\usepackage{amsfonts}
\usepackage{color}
\usepackage{pifont}
\usepackage{upgreek}
\usepackage{bm}
\usepackage{hyperref}
\usepackage{shorttoc}
\usepackage{amsmath,amstext,amsthm,a4,amssymb,amscd}
\usepackage[mathscr]{eucal}
\usepackage{mathrsfs}
\usepackage{epsf}

\numberwithin{equation}{section}

\def\Re{{\rm Re}}

\def\p{\partial}
\def\o{\overline}
\def\b{\bar}
\def\mb{\mathbb}
\def\mc{\mathcal}
\def\n{\nabla}

\def\wt{\widetilde}

\theoremstyle{plain}
\newtheorem{thm}{Theorem}[section]

\newtheorem{prop}[thm]{Proposition}
\newtheorem{cor}[thm]{Corollary}
\theoremstyle{definition}
\newtheorem{rem}[thm]{Remark}

\theoremstyle{definition}

\newcommand{\comment}[1]{}

\usepackage{fancyhdr}
\newenvironment{aligns}{\equation\aligned}{\endaligned\endequation}

\begin{document}

\title{Convexity of energy function associated to the harmonic maps between surfaces}
\author{InKang Kim}
\author{Xueyuan Wan}
\author{Genkai Zhang}

\address{Inkang Kim: School of Mathematics, KIAS, Heogiro 85, Dongdaemun-gu Seoul, 02455, Republic of Korea}
\email{inkang@kias.re.kr}

\address{Xueyuan Wan: School of Mathematics, Korea Institute for Advanced Study, Seoul 02455, Republic of Korea}
\email{xwan@kias.re.kr}

\address{Genkai Zhang: Mathematical Sciences, Chalmers University of Technology and Mathematical Sciences, G\"oteborg University, SE-41296 G\"oteborg, Sweden}
\email{genkai@chalmers.se}

\begin{abstract}
 For a fixed smooth map $u_0$ between two Riemann surfaces $\Sigma$ and $S$ with non-zero degree, we consider the energy function  on Teichm\"uller space $\mc{T}$ of $\Sigma$  that assigns to a complex structure $t\in \mc{T}$ on $\Sigma$ the energy of the harmonic map $u_t:\Sigma_t:=(\Sigma,t) \to S$ homotopic to  $u_0$. We prove that the energy function is convex at its critical points. If $t_0\in\mc{T}$ is a critical point  such that  $du_{t_0}$   is never zero, then the energy function is strictly convex at this  point. As an application, in  the case that $u_0$ is a covering map, we prove that there exists a unique critical point $t_0\in \mc{T}$ minimizing the energy function.  Moreover, the energy density satisfies  $\frac{1}{2}|du|^2(t_0)\equiv 1$ and the Hessian of the energy function is positive definite at this point.

 \end{abstract}
\footnotetext[3]{\sl{Research by Inkang Kim is partially supported by Grant NRF-2019R1A2C1083865 and  research by Genkai Zhang is
  partially supported by Swedish
  Research Council (VR).}}
\maketitle
\tableofcontents

\section*{Introduction}

The energy function on Teichm\"uller space has been widely studied  for the past few decades. In particular, the (strict) convexity and plurisubharmonicity of energy function play a very important role in   Teichm\"uller theory. The convex property of the energy function of harmonic maps on a surface
can be used
to prove that  the Teichm\"uller space
is a cell \cite[Section 3.3]{Tromba},  and the plurisubharmonicity of the energy function of harmonic maps implies that it is a complex Stein manifold
\cite{Bers}, \cite[Section 6]{Wolpert}, \cite[Theorem 6.1.1]{Tromba}, 
and using the (2,0)-part of the pullback of the metric via harmonic maps, it is shown to be homeomorphic to the space of holomorphic quadratic differentials \cite[Theorem 3.1]{Wolf0}.
In this paper, we will continue to consider the (strict) convexity of energy function associated to the harmonic maps between different surfaces. 

For a fixed Riemannian manifold $M$ and a
smooth map (non-constant) $u_0$ from $M$ to a Riemann surface $\Sigma$,
we consider the harmonic maps $u:M\to\Sigma$ in the homotopy class $[u_0]$
and the corresponding  energy $E(u)$
viewed as  a function
 on Teichm\"uller space of 
$\Sigma$. The energy function 
has some  remarkable properties,
such as strict convexity along any Weil-Petersson geodesic
\cite[Theorem 3.1.1]{Yamada},
and the logarithmic strict plurisubharmonicity \cite[Theorem 0.1]{KWZ}.
In particular the energy function has a unique critical point (minimum point) if the induced map between $\pi_1(M)$ and $\pi_1(\Sigma)$ is surjective \cite[Theorem 3.2.1]{Yamada}.
On the other hand we can also consider
the energy function on Teichm\"uller space of $\Sigma$ that assigns to a complex structure on $\Sigma$ the energy of the harmonic map $u:\Sigma \to N$ homotopic to a fixed smooth map $u_0$, where $N$ is a fixed Riemannian manifold.  If $N$ is also a negatively curved Riemann surface and the homotopy class consists of homotopy equivalences,  Tromba \cite[Theorem 6.2.6]{Tromba} showed that this energy function is strictly plurisubharmonic. When $N$ has non-positive Hermitian sectional curvature, Toledo \cite[Theorem 1]{Toledo} proved that the energy function is also plurisubharmonic.
In \cite[Theorem 0.1]{KWZ1}, we proved the reciprocal energy function is plurisuperharmonic, and thus the logarithm of energy function is  plurisubharmonic. 
When  $N$ is $\Sigma$ as a surface
with a fixed complex structure and  $u_0$ is the identity map,
then the energy function has  a unique critical point,
the second derivative of energy function is exactly given by the Weil-Petersson metric \cite[Theorem 3.1.3]{Tromba}. For a general $N$, if  the map induced by $u_0$ between $\pi_1(\Sigma)$ and $\pi_1(N)$ is injective, then there exists a critical point
which  minimizes the energy function, but the uniqueness within the
    homotopy class of $u_0$ fails in general; see \cite{SU1}. 
    
    In this paper, we focus on the case when the target manifold  $N=S$
    is also a fixed Riemann surface with genus $g_S\geq 2$ and
    consider the (strict) convexity of energy function and the uniqueness of the critical point. Our main theorem is 
    \begin{thm}\label{thm0.1}
    Let $u_0:\Sigma\to S$ be a smooth map with non-zero degree, and let $E(t)$ be the associated energy function on the Teichm\"uller space $\mc{T}$ of $\Sigma$. If $t_0\in\mc{T}$ is a critical point of $E(t)$, then the energy function is convex at this point. If moreover the associated harmonic  map $u_{t_0}$ satisfies that  $du_{t_0}$ is never zero, then the energy function is strictly convex at $t_0\in\mc{T}$.
    \end{thm}

\begin{cor}\label{cor0.1}
	Suppose $u_0: \Sigma\to S$ is a covering map, then there exists a unique critical point $t_0\in \mc{T}$ minimizing the energy function $E(t)$.  Moreover, the energy density $\frac{1}{2}|du|^2(t_0)\equiv 1$ and the Hessian of the energy function is positive definite at this point.  Indeed, the hyperbolic structure on $\Sigma$ is the pull-back of the hyperbolic metric on $S$ via $u_{t_0}$.
\end{cor}
We believe that the above corollary holds even for branched coverings with mild assumptions.

We explain briefly our method to prove Theorem \ref{thm0.1} and Corollary \ref{cor0.1}.

For any $t_0\in \mc{T}$ with some abuse of notation
let $\Gamma(t)$, $t\in (-\epsilon,\epsilon)$, be a Weil-Petersson geodesic in $\mc{T}$ passing through $\Gamma(0)=t_0$. The hyperbolic metrics associated to $\Gamma(t)$ can be expressed as
\begin{align*}
g(t)=&\lambda^2_0dzd\b{z}+t(qdz^2+\o{qdz^2})\\
&+\frac{t^2}{2}\left(\frac{2|q|^2}{\lambda_0^4}-2(\Delta-2)^{-1}\frac{2|q|^2}{\lambda_0^4}\right)\lambda^2_0 dzd\b{z}+O(t^4)
\end{align*}
for some holomorphic quadratic differential $qdz^2$ on $\Sigma_{t_0}:=(\Sigma, t_0)$; see \cite[(3.4)]{Wolf}. Here $g_0=\lambda_0^2dzd\bar{z}$ denotes the hyperbolic metric on $\Sigma_{t_0}$. Let $\rho^2dvd\bar{v}$ denote the hyperbolic metric on the fixed Riemann surface $S$. For each $t\in (-\epsilon,\epsilon)$, there exists a unique harmonic map $u_t$ from $(\Sigma, g(t))$ to $(S,\rho^2dvd\bar{v})$  in the homotopy class $[u_0]$. Then the energy function along the Weil-Petersson geodesic is 
\begin{align*}
E(t)=\frac{1}{2}\int_{\Sigma}|du_t|^2d\mu_{g(t)}. 	
\end{align*}
Thus Theorem \ref{thm0.1} would follow
if we can prove that  $d^2 E(t)/dt^2|_{t=0}\geq 0$
at a critical point $t=0$ of $E(t)$, and $
d^2 E(t)/dt^2|_{t=0} >0$ if furthermore $du_{t_0}$ is never zero. For the ease of notations, we use $u_{t_0}=u$ for the rest of the section.
The first derivative of energy function is given by 
\begin{align*}
\frac{d E(t)}{dt}|_{t=0}=-4\left\langle\left(u^*(\rho^2dvd\b{v})\right)^{2,0}, \frac{\b{q}}{\lambda_0^2}\frac{d\b{z}}{dz}\right\rangle_{QB}.	
\end{align*}
Here $\langle\cdot,\cdot\rangle_{QB}$ denotes the natural pairing between holomorphic quadratic differentials and harmonic Beltrami differentials. Thus, $t_0\in\mc{T}$ is a critical point of energy function if and only if $u^*(\rho^2dvd\b{v})\equiv 0$, i.e. $u$ is weakly conformal \cite[Theorem 4]{SU}. The second derivative of energy function is given by 
\begin{align*}
	\frac{1}{2}\frac{d^2 E(t)}{dt^2}|_{t=0}	=\int_{\Sigma}|du|^2\frac{|q|^2}{\lambda_0^4}d\mu_{g_0}-\left\langle \mc{J}\left(\frac{\p u}{\p t}\right),\frac{\p u}{\p t}\right\rangle,
\end{align*}
where $d\mu_{g_0}:=\lambda_0^2\frac{i}{2}dz\wedge d\b{z}$ and $\mc{J}$ is Jacobi operator defined in \eqref{2.21}, which is real, symmetric and semi-positive, $\langle\cdot,\cdot\rangle$ is the Hermitian inner product  on the bundle $u^*T_{\mb{C}}S$ defined by \eqref{HI1} and \eqref{HI2}. We denote by $\n$ the natural induced connection on $u^*T_{\mb{C}}S$ from the Riemann surface $(S, \rho^2dvd\bar{v})$, and $\n^{0,1}$ denotes the $(0,1)$-part of $\n$, the adjoint operator of $\n^{0,1}$ is denoted by $(\n^{0,1})^*$ and set 
$\Delta^{0,1}=\n^{0,1}(\n^{0,1})^*+(\n^{0,1})^*\n^{0,1}$, then 
\begin{align*}
\mc{J}=\Delta^{0,1}+\mc{R},	
\end{align*}
where $\mc{R}$ is a semi-positive operator. Thus 
\begin{align*}
\frac{d^2 E(t)}{dt^2}|_{t=0}\geq 4\left(\|\mb{H}(i_{\mu}du)\|^2-\text{Re}\langle (\n^{0,1})^*i_{\mu}du,\mc{J}^{-1}(\o{ (\n^{0,1})^*i_{\mu}du})\rangle\right),	
\end{align*}
	with the equality if and only if $\mc{R}(\mc{J}^{-1}(\n^{0,1})^*i_{\mu}du)=0$. Here $\mu:=\frac{\b{q}}{\lambda_0^2}\frac{d\b{z}}{dz}$ is a harmonic Beltrami differential, $\mb{H}$ denotes the harmonic projection to $\text{Ker} \Delta^{0,1}$. If $t_0\in \mc{T}$ is a critical point, then $u_{t_0}$ is (anti)-holomorphic and 
	\begin{align*}
	\mc{R}(\mc{J}^{-1}(\n^{0,1})^*i_{\mu}du)=0,\quad 	\text{Re}\langle (\n^{0,1})^*i_{\mu}du,\mc{J}^{-1}(\o{ (\n^{0,1})^*i_{\mu}du})\rangle=0.
	\end{align*}
Therefore, 
$$\frac{d^2 E(t)}{dt^2}|_{t=0}= 4\|\mb{H}(i_{\mu}du)\|^2\geq 0,$$
i.e. the energy function is convex at its critical points. If moreover, $du_{t_0}$ is never zero, then we can prove 
$\mb{H}(i_{\mu}du)\not\equiv 0$, so the energy function is strictly convex at this point. Hence we get Theorem \ref{thm0.1}. 

If $u_0:\Sigma\to S$ is a covering map, then $u_0$ is a surjective map with $\deg u_0\neq 0$, and for any $p\in \Sigma$, the induced  homomorphism $(u_0)_*: \pi_1(\Sigma, p)\to \pi_1(S, u_0(p))$ is injective. By \cite{SU1,SY}, there exists a critical point minimizing the energy function. 
 If $t_0\in\mc{T}$ is a critical point, then $u_{t_0}$ is weakly conformal and so $\pm$ holomorphic. By Riemann-Hurwitz formula, 
 \begin{align*}
\chi(\Sigma)=\deg u_{t_0}\cdot\chi(S)-\sum(d_i-1)=\deg u_0\cdot\chi(S)-\sum(d_i-1),	
\end{align*}
where $d_i\geq 1$ is the ramification index of  $u_{t_0}$; see e.g. \cite[Theorem 1.2 or (1.1)]{KLP}.
On the other hand, since $u_0$ is a covering map, so $$\chi(\Sigma)=\deg u_0 \cdot\chi(S).$$  Thus, each ramification index $d_i=1$, which implies that $du_{t_0}$ is never zero. From Theorem \ref{thm0.1}, the energy function is strictly convex at its critical points. By \cite[Lemma 4.11 (2)]{JE} and noting that $\mc{T}$ is connected, the critical point is unique and this
proves Corollary \ref{cor0.1}.

This article is organized as follows. In Section \ref{sec1}, we fix the notation and recall
some basic facts on  the energy, harmonic maps between surfaces, and Weil-Petersson geodesics on Teichm\"uller space. In Section \ref{sec2}, we will calculate the first and the second derivatives of energy function along a Weil-Petersson geodesic. In Section \ref{sec3}
we will prove the (strict) convexity of energy function at its critical point and study
the uniqueness of its critical pints,  and we prove
Theorem \ref{thm0.1} and Corollary \ref{cor0.1}.

\section{Preliminaries}\label{sec1}

In this section, we shall fix the notation and recall some basic facts on the energy, harmonic maps between surfaces, and Weil-Petersson geodesics on Teichm\"uller space. We refer \cite{DW, Jost, Jost1, Tromba, Wolf, Wolpert} for more details.   

\subsection{Energy function on Teichm\"uller space}

Let $\Sigma$ be a Riemann surface of genus $g_{\Sigma}\geq 2$, and $z=x+ i y$ denote the  local holomorphic coordinate on $\Sigma$. Denote  by 
\begin{align*}
	\lambda^2(z) dzd\b{z}:=\frac{\lambda^2}{2}(dz\otimes d\b{z}+d\b{z}\otimes dz)=\lambda^2(dx^2+dy^2)
\end{align*}
the hyperbolic metric, i.e. its curvature satisfies 
 \begin{align*}
 K:=-\frac{4}{\lambda^2}\frac{\p^2}{\p z\p\b{z}}\log\lambda=-1.	
 \end{align*}
The associated Hermitian metric is $\lambda^2 dz\otimes d\b{z}$, and the fundamental $(1,1)$-form is given by
\begin{align}\label{fund}
\omega_{\Sigma}=\frac{i}{2}\lambda^2 dz\wedge d\b{z}=\lambda^2 dx\wedge dy,
\end{align}
where $dz\wedge d\b{z}=dz\otimes d\b{z}-d\b{z}\otimes dz$.
The area of $\Sigma$ is 
\begin{align}\label{Area}
\text{Area}(\Sigma):=\int_{\Sigma}\omega_{\Sigma}=-2\pi\chi(\Sigma),
\end{align}
where the last equality holds by Gauss-Bonnet theorem,  $\chi(\Sigma)=2-2g_{\Sigma}$ is the  Euler characteristic of $\Sigma$. 

Let $S$ be also a Riemann surface of genus $g_{S}\geq 2$, and equipped with the hyperbolic metric 
\begin{align*}
\rho^2(v)dvd\b{v}=\frac{\rho^2}{2}(dv\otimes d\b{v}+d\b{v}\otimes dv),
\end{align*}
where $v$ denotes the local holomorphic coordinate on $S$. Thus
\begin{align}\label{1.3}
	\frac{4}{\rho^2}\frac{\p^2}{\p v\p\b{v}}\log \rho=1.
\end{align}
With respect to the local complex coordinates $\{v,\b{v}\}$ on $S$,
any smooth map $u:\Sigma\to S$ can be written
as $$(v, \b{v})=u(z)=(u^v,u^{\b{v}})$$
 for some local smooth functions $u^v, u^{\b{v}}$.
 The energy of $u$ is defined
 \cite[(3.6.1)]{Jost}
 by 
\begin{align}\label{energy}
E(u):=\int_{\Sigma} \rho^2(u(z))(|u^v_z|^2+|u^v_{\b{z}}|^2)\frac{i}{2}dz\wedge d\b{z},
\end{align}
where $u_z^v:=\p_z u^v$ and $u^v_{\b{z}}:=\p_{\b{z}}u^v$.
$u$ is called harmonic if  it is a critical point of the energy $E$, i.e., if  it satisfies the following harmonicity  equation:
\begin{align}\label{harmonic}
\n_{\b{z}}u^v_z:=\p_{\b{z}}u^v_z+\frac{2\rho_v}{\rho} u^v_z u^v_{\b{z}}=0,	
\end{align}
where $\n$ denotes the natural induced connection on the complexified bundle 
$$u^*T_{\mb{C}}S=:u^*(TS\otimes\mb{C})=u^*K_{S}^*\oplus \o{u^*K_{S}^*}$$
from the Chern connection of the Hermitian line bundle $(K_S^*, \rho^2dv\otimes d\b{v})$.

Fix a smooth map $u_0:\Sigma\to S$ with non-zero degree, i.e. $\deg u_0\neq 0$. Then any smooth map $u:\Sigma\to S$ in the homotopy class $[u_0]$ has the same degree $\deg u=\deg u_0\neq 0$, and so $u$ is surjective. Let $$\mc{T}=\mc{M}_{-1}/\mc{D}_0$$ be the Teichm\"uller space of $\Sigma$,
where $\mc{M}_{-1}$ is the space of all smooth Riemannian metrics on $\Sigma$ with scalar curvature $-1$ and $\mc{D}_0$
is the group of all smooth orientation preserving diffeomorphisms of $\Sigma$ in the identity homotopy class.  For the fixed Riemann surface $S$ and any  hyperbolic metric $g$ in $\mc{M}_{-1}$, since each smooth map homotopic to $u_0$ is surjective, so there exists a unique harmonic map $u(g)$ homotopy to $u_0$; see e.g. \cite[Page 158, Corollary]{ES}, \cite[Page 675, (I)]{Hart}, \cite[Thoerem 4]{Sampson} or also \cite[Theorem 3.1.1]{Tromba}. Thus the energy function 
\begin{align*}
	E(g):=E(u(g))
\end{align*}
is a smooth function on $\mc{M}_{-1}$. By the same argument as in \cite[Page 66-67]{Tromba}, 
\begin{align*}
	E(f^*g)=E(g),
\end{align*}
for any $f\in\mc{D}_0$. Thus $E$ descends to a smooth function on Teichm\"uller space $\mc{T}:=\mc{M}_{-1}/\mc{D}_0$. We denote the energy function by $E(t)$, $t\in \mc{T}$.

\subsection{Weil-Petersson geodesic}

For any $t\in \mc{T}$, denote by $\Sigma_t=(\Sigma,t)$ the associated Riemann surface with the complex structure $t\in\mc{T}$. Let $\omega_{\Sigma_t}$ be the fundamental $(1,1)$-form as in (\ref{fund}).
The tangent space $T_t\mc{T}$ is identified with the space of harmonic Beltrami differential $\mu=\mu(z)\frac{d\b{z}}{d z}$, and the $L^2$-norm defines the Weil-Petersson metric
\begin{align}\label{WP}
\|\mu\|^2_{\text{WP}}=\int_{\Sigma}|\mu(z)|^2\omega_{\Sigma_t}.	
\end{align}
The Weil-Petersson metric is K\"ahler \cite{Ahlfors},  negatively curved \cite{Tromba1, Wolpert2}, not complete \cite{Chu, Wolpert1}.  However, the
synthetic geometry is quite similar to that of a complete metric of negative curvature, and indeed  Wolpert \cite{Wolpert} showed that every pair of points can be joined by a unique Weil-Petersson geodesic, and a geodesic length function is strictly convex along a Weil-Petersson geodesic. 

Fix a point $t_0\in \mc{T}$ and  let
$g_0=\lambda_0^2 dzd\b{z}$ be
the corresponding hyperbolic metric on $\Sigma_{t_0}$.
Let $\Gamma(t)$, ($\Gamma(0)=t_0$),  be the Weil-Petersson geodesic arc with initial tangent
vector given by the harmonic Beltrami differential
$\mu=\frac{\b{q}}{\lambda_0^2}\frac{d\b{z}}{dz}$, where $qdz^2$ is a
holomorphic quadratic differential on $\Sigma_{t_0}$.
Then the
associated hyperbolic metrics on $\Sigma_t$ has the following Taylor expansion
near $t=0$ as \cite[(3.4)]{Wolf},
\begin{aligns}\label{1.1}
g(t)=&\lambda^2_0dzd\b{z}+t(qdz^2+\o{qdz^2})\\
&+\frac{t^2}{2}\left(\frac{2|q|^2}{\lambda_0^4}-2(\Delta-2)^{-1}\frac{2|q|^2}{\lambda_0^4}\right)\lambda^2_0 dzd\b{z}+O(t^4).
\end{aligns}
Here $\Delta=\frac{4}{\lambda_0^2}\frac{\p^2}{\p
  z\p\b{z}}=\frac{1}{\lambda_0^2}(\p_x^2+\p^2_y)$.
Furthermore there is
\cite[Lemma 5.1]{Wolf}
a point-wise estimate for the term involving
$(\Delta-2)^{-1}$,
\begin{align}\label{1.2}
	\alpha:=-2(\Delta-2)^{-1}\frac{|q|^2}{\lambda_0^4}\geq \frac{1}{3}\frac{|q|^2}{\lambda_0^4}.
\end{align}

\section{Variations of energy function}\label{sec2}

In this section, we will calculate the first and the second derivatives of energy function along a Weil-Petersson geodesic. 

\subsection{Energy function along Weil-Petersson geodesics}

From (\ref{1.1}) and (\ref{1.2}) the Riemannian metric $g(t)$
can be written as
 \begin{align*}
 g(t)=(dz, d\b{z})\otimes G\left(\begin{matrix}{}
 dz \\
d\b{z}	
\end{matrix}\right), 	
 \end{align*}
where $G$ is a matrix given by 
\begin{align*}
G=\left(\begin{matrix}{}
 G_{zz} & G_{z\b{z}}\\
 G_{z\b{z}} & G_{\b{z}\b{z}}	
\end{matrix}
\right)	=\left(\begin{matrix}{}
 tq  & \frac{\lambda^2_0}{2}+\frac{t^2}{2}\left(\frac{|q|^2}{\lambda^4_0}+\alpha\right)\lambda^2_0\\
 \frac{\lambda^2_0}{2}+\frac{t^2}{2}\left(\frac{|q|^2}{\lambda^4_0}+\alpha\right)\lambda^2_0 & t\o{q}   
\end{matrix}
\right)+O(t^4).
\end{align*}
Thus 
\begin{aligns}\label{vol1}
\det G &=t^2|q|^2-\left(\frac{\lambda^2_0}{2}+\frac{t^2}{2}(\frac{|q|^2}{\lambda^4_0}+\alpha)\lambda^2_0\right)^2+O(t^4)\\
&=-\frac{\lambda_0^4}{4}-\frac{t^2}{2}(\alpha\lambda_0^4-|q|^2)+O(t^4).	
\end{aligns}
The volume element $d\mu_{g(t)}$ is  given by 
\begin{align}\label{vol2}
d\mu_{g(t)}:=i\sqrt{|\det G|} dz\wedge d\b{z}.	
\end{align}
Fix a smooth map $u:\Sigma\to S$  with  $\deg u_0\neq 0$. For each $t$, we get a harmonic map $u=u_t: (\Sigma, g(t))\to (S, \rho^2dvd\b{v})$
which is homotopic to a fixed smooth map $u_0$. The energy density  is defined by
\begin{align}\label{density}
\frac{1}{2}|du|^2:=\frac{1}{2}\text{Tr}_{g(t)}(u^*(\rho^2dvd\b{v}))=\frac{1}{2}\text{Tr}(G^{-1}U),	
\end{align}
where 
\begin{align*}
du:=u^v_zdz\otimes\frac{\p}{\p v}+u^v_{\b{z}}d\b{z}\otimes \frac{\p}{\p v}+\o{u^v_{\b{z}}}dz\otimes \frac{\p}{\p\b{v}}+\o{u^v_z}	d\b{z}\otimes \frac{\p}{\p \b{v}}\in A^1(\Sigma, u^*T_{\mb{C}}S)
\end{align*}
and the matrix $U$ is
\begin{align*} 
U=\rho^2\left(\begin{matrix}{}
 u^v_z\o{u^v_{\b{z}}} & \frac{1}{2}(|u^v_z|^2+|u^v_{\b{z}}|^2)\\
 \frac{1}{2}(|u^v_z|^2+|u^v_{\b{z}}|^2) & u^v_{\b{z}}\o{u^v_z}	
\end{matrix}
\right).
\end{align*}
It can be also written as
(\ref{density}), 
 \begin{align}\label{density1}
 	\frac{1}{2}|du|^2 =\frac{1}{2}\frac{\rho^2(u(z))}{\det G}\left(G_{zz}u^v_{\b{z}}\o{u^v_z}+G_{\b{z}\b{z}}u^v_z\o{u^v_{\b{z}}}-G_{z\b{z}}(|u^v_z|^2+|u^v_{\b{z}}|^2)\right).
 \end{align}
 Hence, the energy function  \cite[(1.2.3), (1.2.4)]{Xin} 
 along the Weil-Petersson geodesic $\Gamma(t)$ is given by
\begin{align}\label{energy function}
E(t)=\frac{1}{2}\int_{\Sigma}|du|^2 d\mu_{g_{t}},
\end{align}
with $u=u(t)$ depending on $t$. 
\begin{rem}\label{rem1}
At $t=0$ the matrix $G$ is
\begin{align*}
G_{zz}=G_{\b{z}\b{z}}=0,\quad G_{z\b{z}}=\frac{\lambda^2_0}{2},\quad \det G=-\frac{\lambda^4_0}{4}.	
\end{align*}
	So 
	\begin{align}\label{2.0}
	|du|^2=	\frac{2\rho^2}{\lambda^2_0}(|u^v_z|^2+|u^v_{\b{z}}|^2),\quad d\mu_{g_0}=\lambda_0^2\frac{i}{2}dz\wedge d\b{z},
	\end{align}
and then the energy at $t=0$ is 
\begin{align*}
E(0)=\frac{1}{2}\int_{\Sigma}|du|^2 d\mu_{g_{0}}=\int_{\Sigma}\rho^2(|u^v_z|^2+|u^v_{\b{z}}|^2) \frac{i}{2}dz\wedge d\b{z},	
\end{align*}
which is same as (\ref{energy}).
\end{rem}

\subsection{The first derivative}
Now we calculate the first derivative of energy function $E(t)$ at
$t=0$. The engergy  (\ref{energy function}) is
\begin{align}\label{first}
	\frac{d E(t)}{dt}|_{t=0}=\frac{1}{2}\int_{\Sigma} \frac{\p |du|^2}{\p t}|_{t=0}d\mu_{g_0}+\frac{1}{2}\int_{\Sigma}|du|^2(0)\frac{\p d\mu_{g_t}}{\p t}|_{t=0}.
\end{align}
By the expression of $d\mu_{g(t)}$, (\ref{vol1}) and (\ref{vol2}), we have
\begin{align}\label{Vol3}
	\frac{\p d\mu_{g_t}}{\p t}|_{t=0}=idz\wedge d\b{z}\cdot \frac{1}{2\sqrt{|\det G|}}(-\frac{\p}{\p t}\det G)|_{t=0}=0.
\end{align}
 So the second term in the RHS of (\ref{first}) vanishes and 
\begin{align}\label{first1}
  \frac{d E(t)}{dt}|_{t=0}=\frac{1}{2}\int_{\Sigma}
  \frac{\p |du|^2}{\p t}|_{t=0}d\mu_{g_0}.
\end{align}
To find the integrand $\frac{\p |du|^2}{\p t}|_{t=0}$ we
use (\ref{density1}) and the formulas for $G$. We have
\begin{align*}
	\frac{\p G_{z\b{z}}}{\p t}|_{t=0}=\frac{\p\det G}{\p t}|_{t=0}=0,\quad G_{zz}=tq,
\end{align*}
and 
\begin{align}\label{2.1}
\frac{1}{2}\frac{\p |du|^2}{\p t}|_{t=0}=\text{Re}\left(-\frac{4\rho^2}{\lambda^4_0}\b{q}\o{u^v_{\b{z}}}u^v_z\right)+\frac{1}{\lambda_0^2}\frac{\p}{\p t}\left(\rho^2(|u^v_z|^2+|u^v_{\b{z}}|^2)\right).
\end{align}
Since $u_t$ is a harmonic map, i.e. a critical point of energy functional $E(u)$, so
\begin{multline}\label{2.2}
\int_{\Sigma}	\frac{1}{\lambda_0^2}\frac{\p}{\p t}\left(\rho^2(|u^v_z|^2+|u^v_{\b{z}}|^2)\right)d\mu_{g_0}=\int_{\Sigma}	\frac{\p}{\p t}\left(\rho^2(|u^v_z|^2+|u^v_{\b{z}}|^2)\right)\frac{i}{2}dz\wedge d\b{z}=0.
\end{multline}
Substituting (\ref{2.1}) into (\ref{first1}) and using (\ref{2.2}), we obtain
\begin{multline}\label{2.3}
	\frac{d E(t)}{dt}|_{t=0}=\text{Re}\int_{\Sigma}\left(-\frac{4\rho^2}{\lambda^2_0}\b{q}\o{u^v_{\b{z}}}u^v_z\right)\frac{i}{2}dz\wedge d\b{z}=-4\left\langle\rho^2u^v_z\o{u^v_{\b{z}}}dz^2, \frac{\b{q}}{\lambda_0^2}\frac{d\b{z}}{dz}\right\rangle_{QB},
\end{multline}
where $\langle\cdot,\cdot\rangle_{QB}$ is a pairing between holomorphic quadratic differentials and harmonic Beltrami differentials by 
$$\langle \phi dz^2,\mu\frac{d\b{z}}{dz}\rangle_{QB}:=\text{Re}\int_{\Sigma}\phi\mu\frac{i}{2}dz\wedge d\b{z},$$
which is non-degenerate. Note here that $\rho^2u^v_z\o{u^v_{\b{z}}}dz^2$ is a holomorphic quadratic differential (called {\it Hopf differential of $u$}) since
\begin{align*}
	\p_{\b{z}}(\rho^2u^v_z\o{u^v_{\b{z}}})=\rho^2\n_{\b{z}}u^v_z\o{u^v_{\b{z}}}+\rho^2u^v_z\o{\n_zu^v_{\b{z}}}=0
\end{align*}
by the harmonicity equation (\ref{harmonic}). Thus if $t=0$ is a critical point of the energy function, then 
(\ref{2.3}) is equal to zero for any harmonic Beltrami differential, which implies that 
\begin{align*}
\left(u^*(\rho^2dvd\b{v})\right)^{2,0}=	\rho^2u^v_z\o{u^v_{\b{z}}}dz^2=0,
\end{align*}
i.e. the harmonic map $u$ for $t=0$ is weakly conformal, and so 
 $u$ is $\pm$ holomorphic, see e.g. \cite[Page 87 and Proposition 3.5.9]{BC}. In fact, the $(0,1)$-part of the natural connection $\n^{u^*K_{S}^*\otimes K_{\Sigma}}$ gives a holomorphic structure on the bundle $u^*K_{S}^*\otimes K_{\Sigma}$, and the harmonicity equation (\ref{harmonic}) means that $u^v_z\frac{\p}{\p v}\otimes dz$ is a holomorphic section of $u^*K_{S}^*\otimes K_{\Sigma}$, thus the set of zero points of $u^v_z$ is discrete unless $u^v_z\equiv 0$. Similarly, the set of zero points of $u^v_{\b{z}}$ is also discrete unless $u^v_{\b{z}}\equiv 0$. Thus $u^v_z\o{u^v_{\b{z}}}\equiv 0$ if and only if $u$ is $\pm$ holomorphic.
Therefore,
\begin{prop}\label{prop0}
$t_0\in \mc{T}$ is a critical point of energy function if and only if the associated harmonic map $u_{t_0}$ is $\pm$ holomorphic. 	
\end{prop}
Indeed, the above proposition was contained in Sack-Uhlenbeck {\cite[Theorem 4]{SU}} by considering the harmonic maps from Riemann surfaces to a general Riemannian manifold, see also \cite[Theorem 1.8]{SU2} for its proof.

\subsection{The second derivative}

The second derivative of energy function, by \eqref{Vol3},
is
\begin{align}\label{second}
	\frac{d^2 E(t)}{dt^2}|_{t=0}=\frac{1}{2}\int_{\Sigma} \frac{\p^2 |du|^2}{\p t^2}|_{t=0}d\mu_{g_0}+\frac{1}{2}\int_{\Sigma}|du|^2(0)\frac{\p^2 d\mu_{g_t}}{\p t^2}|_{t=0}.
\end{align}
The second term in the RHS, by \eqref{vol1} and \eqref{2.0}, is
\begin{multline}\label{2.9}
\quad \frac{1}{2}\int_{\Sigma}|du|^2(0)	 \frac{1}{2\sqrt{|\det G|}}(-\frac{\p^2}{\p t^2}\det G)|_{t=0}idz\wedge d\b{z}\\
=\int_{\Sigma}	\frac{2\rho^2}{\lambda^2_0}(|u^v_z|^2+|u^v_{\b{z}}|^2)\frac{1}{\lambda^2_0}(\alpha\lambda_0^4-|q|^2) \frac{i}{2}dz\wedge d\b{z}.
\end{multline}
To deal with the first term in the RHS of \eqref{second}
we use \eqref{density1} to find
\begin{aligns}\label{2.10}
\frac{1}{2}\frac{\p^2 |du|^2}{\p t^2}|_{t=0} &=\frac{\p^2}{\p t^2}\left(\text{Re}\left(\frac{G_{zz}}{\det G}\cdot\rho^2 u^v_{\b{z}}\o{u^v_z}\right)-\frac{1}{2}\frac{G_{z\b{z}}}{\det G}\cdot\rho^2(|u^v_z|^2+|u^v_{\b{z}}|^2)\right)\\
&=2\text{Re}\left(\frac{\p}{\p t}\left(\frac{G_{zz}}{\det G}\right)\frac{\p}{\p t}(\rho^2 u^v_{\b{z}}\o{u^v_z})\right)-\frac{1}{2}\frac{\p^2}{\p t^2}\left(\frac{G_{z\b{z}}}{\det G}\right)\cdot\rho^2(|u^v_z|^2+|u^v_{\b{z}}|^2)\\
&\quad -\frac{1}{2}\frac{G_{z\b{z}}}{\det G}\frac{\p^2}{\p t^2}(\rho^2(|u^v_z|^2+|u^v_{\b{z}}|^2))\\
&=\text{Re}\left(-\frac{8q}{\lambda_0^4}\frac{\p}{\p t}(\rho^2 u^v_{\b{z}}\o{u^v_z})\right)+ \frac{2}{\lambda^2_0}\left(-\alpha+\frac{3|q|^2}{\lambda_0^4}\right)\cdot\rho^2(|u^v_z|^2+|u^v_{\b{z}}|^2)\\
&\quad +\frac{1}{\lambda_0^2}\frac{\p^2}{\p t^2}(\rho^2(|u^v_z|^2+|u^v_{\b{z}}|^2)),
\end{aligns}
where the last two equalities follow from 
\begin{align*}
\frac{G_{zz}}{\det G}=\frac{t q}{-\frac{\lambda_0^4}{4}-\frac{t^2}{2}(\alpha\lambda_0^4-|q|^2)+O(t^4)}=-\frac{4q}{\lambda_0^4}t+O(t^3)
\end{align*}
and
\begin{align*}
-\frac{G_{z\b{z}}}{\det G}=\frac{\frac{\lambda^2_0}{2}+\frac{t^2}{2}\left(\frac{|q|^2}{\lambda^4_0}+\alpha\right)\lambda^2_0+O(t^4)}{\frac{\lambda^4_0}{4}+\frac{t^2}{2}(\alpha\lambda_0^4-|q|^2)+O(t^4)}	=\frac{2}{\lambda^2_0}+\frac{t^2}{2}\cdot \frac{4}{\lambda^2_0}\left(-\alpha+\frac{3|q|^2}{\lambda_0^4}\right)+O(t^4).
\end{align*}
Fix a small $\epsilon>0$. Write
the smooth family of harmonic maps $\{u_t\}_{t\in 
  (-\epsilon,\epsilon)}$
as $u$, 
\begin{align*}
u: (-\epsilon,\epsilon)\times\Sigma\to S,\quad u(t,z):=u_t(z).
\end{align*}
The pullback complexified tangent
bundle $$u^*T_{\mb{C}}S=u^*(K_S^*\oplus \o{K_S^*})\to
(-\epsilon,\epsilon)\times\Sigma$$ endows with an induced connection
from the Chern connection of the Hermitian line bundle
$(K_S^*,\rho^2dv\otimes d\b{v})$, and we denote it also by $\n$. Then 
\begin{align}\label{2.11}
\frac{\p}{\p t}(\rho^2u^v_z\o{u^v_{\b{z}}})=\rho^2(\n_tu^v_z)	\o{u^v_{\b{z}}}+\rho^2u^v_z\o{\n_tu^v_{\b{z}}},
\end{align}
and
\begin{aligns}\label{2.4}
  \frac{\p^2}{\p t^2}\left(\rho^2(|u^v_z|^2+|u^v_{\b{z}}|^2)\right) & =
  \frac{\p}{\p t}\left(\rho^2\left( (\n_tu^v_z) \o{u^v_z}+u^v_z\o{\n_tu^v_z}+\n_t u^v_{\b{z}}\o{u^v_{\b{z}}}+u^v_{\b{z}}\o{\n_t u^v_{\b{z}}}\right)\right)\\
&=2\rho^2\left(\text{Re}( (\n_t\n_z u^v_t) \o{u^v_z})+|\n_tu^v_z|^2\right)\\
&\quad +2\rho^2\left(\text{Re}( (\n_t\n_{\b{z}} )u^v_t \o{u^v_{\b{z}}})+|\n_tu^v_{\b{z}}|^2\right).
\end{aligns}
Since $[\n_t,\n_z]=R(\frac{\p u}{\p t},\frac{\p u}{\p z})$,
where $R=\n^2$ is the curvature operator, so
\begin{aligns}\label{2.5}
\n_t\n_zu^v_t &=\n_z\n_tu^v_t+R^v_{vtz}u^v_t\\
&=\n_z\n_tu^v_t+R^{v}_{vv\b{v}}(u^v_t\o{u^v_{\b{z}}}-\o{u^v_t}u^v_z)u^v_t\\
&=\n_z\n_tu^v_t-\frac{\rho^2}{2}(u^v_t\o{u^v_{\b{z}}}-\o{u^v_t}u^v_z)u^v_t.
\end{aligns}
Here  we have used the fact that
\begin{align*}
R^v_{vtz}\frac{\p}{\p v}&:=R(\frac{\p u}{\p t},\frac{\p u}{\p z})\frac{\p}{\p v}=u^v_t\o{u^v_{\b{z}}}R(\frac{\p}{\p v},\frac{\p}{\p\b{v}})\frac{\p}{\p v} +\o{u^v_t}u^v_zR(\frac{\p}{\p\b{v}},\frac{\p}{\p v})\frac{\p}{\p v}\\
&=(u^v_t\o{u^v_{\b{z}}}-\o{u^v_t}u^v_z)R^{v}_{vv\b{v}}\frac{\p}{\p v},
\end{align*}
with $R^{v}_{vv\b{v}}\frac{\p}{\p v}:=R(\frac{\p}{\p
  v},\frac{\p}{\p\b{v}})\frac{\p}{\p v}, $
and the fact \eqref{1.3} that
$$R^v_{vv\b{v}}=-\p_{\b{v}}(\p_v\log \rho^2)=-\frac{\rho^2}{2}.
$$
Similarly, 
\begin{align}\label{2.6}
\n_t\n_{\b{z}}u^v_t=\n_{\b{z}}\n_t u^v_t-\frac{\rho^2}{2}(u^v_t\o{u^v_z}-\o{u^v_t}u^v_{\b{z}})u^v_t.	
\end{align}
Thus \eqref{2.4} becomes
\begin{aligns}\label{2.7}
\frac{\p^2}{\p t^2}\left(\rho^2(|u^v_z|^2+|u^v_{\b{z}}|^2)\right) & = 2\rho^2\text{Re}(\n_z\n_tu^v_t\o{u^v_z}+\n_{\b{z}}\n_tu^v_t\o{u^v_{\b{z}}})\\
&\quad+2\rho^2(|\n_tu^v_z|^2+|\n_tu^v_{\b{z}}|^2)\\
&\quad+\rho^4\left(-2\text{Re}((u^v_t)^2\o{u^v_z}\o{u^v_{\b{z}}})+|u^v_t|^2(|u^v_z|^2+|u^v_{\b{z}}|^2)\right).
\end{aligns}
We integrate the  first term in the RHS of (\ref{2.7}) above using
 \eqref{harmonic} and integration by parts,
\begin{aligns}\label{2.8}
  &\quad \int_{\Sigma}\frac{1}{\lambda_0^2}	2\rho^2\text{Re}(
  (\n_z\n_tu^v_t )\o{u^v_z}+\n_{\b{z}}\n_tu^v_t\o{u^v_{\b{z}}})d\mu_{g_0}\\
&=\text{Re}\int_{\Sigma}\rho^2( (\n_z\n_tu^v_t) \o{u^v_z}+\n_{\b{z}}\n_tu^v_t\o{u^v_{\b{z}}})idz\wedge d\b{z}\\
&=-\text{Re}\int_{\Sigma}\rho^2(\n_tu^v_t\o{\n_{\b{z}}u^v_z}+\n_tu^v_t\o{\n_zu^v_{\b{z}}})idz\wedge d\b{z}=0.
\end{aligns}
Thus,
altogether
using  \eqref{second}-\eqref{2.11}, \eqref{2.7} and \eqref{2.8} 
we obtain
   \begin{aligns}\label{second1}
   	\frac{d^2 E(t)}{dt^2}|_{t=0} &=\text{Re}\int_{\Sigma}-\frac{4\b{q}}{\lambda_0^2}\rho^2(\n_t u^v_z\o{u^v_{\b{z}}}+u^v_z\o{\n_tu^v_{\b{z}}})idz\wedge d\b{z}\\
   	&\quad+\int_{\Sigma}2\rho^2(|u^v_z|^2+|u^v_{\b{z}}|^2)\frac{|q|^2}{\lambda_0^4}idz\wedge d\b{z}\\
   	&\quad+\int_{\Sigma}\rho^2(|\n_t u^v_z|^2+|\n_tu^v_{\b{z}}|^2)idz\wedge d\b{z}\\
   	&\quad+\int_{\Sigma}\rho^4\left(-2\text{Re}((u^v_t)^2\o{u^v_z}\o{u^v_{\b{z}}})+|u^v_t|^2(|u^v_z|^2+|u^v_{\b{z}}|^2)\right)\frac{i}{2}dz\wedge d\b{z}.
   \end{aligns}
Note again that $u_t$ is a harmonic map for each $t\in (-\epsilon,\epsilon)$, so it satisfies the following harmonicity equation 
\begin{align}\label{2.12}
\text{Tr}_{g(t)}(\n^{u^*T_{\mb{C}}S\otimes K_{\Sigma}}du)=0	
\end{align}
where $\n^{u^*T_{\mb{C}}S\otimes K_{\Sigma}}$ denote the induced the connection from the connection $\n$ on $u^*T_{\mb{C}}S$ and the Chern connection on the Hermitian line bundle $(K_{\Sigma}, \frac{1}{\lambda_0^2}\frac{\p}{\p z}\otimes \frac{\p}{\p\b{z}})$; see e.g. \cite[Definition 1.2.2]{Xin}. 
 Denote $F=u^*K_S^*\otimes K_{\Sigma}$ and let $\n^F$ be the induced connection on $F$ so that $\n^{u^*T_{\mb{C}}S\otimes K_{\Sigma}}=\n^F\oplus\o{\n^F}$.  In local coordinates, \eqref{2.12} is equivalent to 
\begin{align}\label{2.13}
G^{zz}\n^F_zu^v_z+G^{z\b{z}}\n^F_z u^v_{\b{z}}+G^{\b{z}z}\n^F_{\b{z}}u^v_z+G^{\b{z}\b{z}}\n^F_{\b{z}}u^v_{\b{z}}=0.	
\end{align}
Taking the derivative of both sides of \eqref{2.13} at $t=0$ results in
 \begin{align}\label{2.15}
 \frac{\b{q}}{\lambda^2_0}\n^F_zu^v_z+\frac{q}{\lambda^2_0}\n^F_{\b{z}}u^v_{\b{z}}=\frac{\p}{\p t}\n^F_zu^v_{\b{z}}.	
 \end{align}
Let $\{\Gamma^{z}_{zz}, \Gamma^z_{z\b{z}}, \Gamma^{\b{z}}_{zz}\}$ denote the Christoffel symbols of the hyperbolic metrics $g(t)$. Since $G_{zz}=tq$,$G_{\b{z}\b{z}}=t\b{q}$ and $q$ is holomorphic, so 
\begin{align}\label{2.14}
\Gamma^z_{z\b{z}}(t)=\frac{1}{2}G^{zz}\frac{\p G_{zz}}{\p\b{z}}+\frac{1}{2}G^{z\b{z}}\frac{\p G_{\b{z}\b{z}}}{\p z}=0,	
\end{align}
and 
\begin{align*}
\Gamma^z_{zz}(0)=\p_z\log\lambda_0^2,\quad 	\Gamma^{\b{z}}_{zz}(0)=0.
\end{align*}
This implies that for all $t$
\begin{align}\label{2.16}
\n^F_zu^v_{\b{z}}=\n_{z}u^v_{\b{z}}-\Gamma^{z}_{z\b{z}}u^v_z-\Gamma^{\b{z}}_{z\b{z}}u^v_{\b{z}}=\n_{z}u^v_{\b{z}}.
\end{align}
On the other hand, by \eqref{harmonic} and \eqref{2.5} we have
\begin{aligns}\label{2.17}
\frac{\p}{\p t}\n_zu^v_{\b{z}}&=\n_t \n_z u^v_{\b{z}}
=\frac{1}{2}(\n_t\n_z u^v_{\b{z}}+\n_t\n_{\b{z}}u^v_z)\\
&=\frac{1}{2}(\n_z\n_t u^v_{\b{z}}+\n_{\b{z}}\n_t u^v_z+R^v_{vtz}u^v_{\b{z}}+R^v_{vt\b{z}}u^v_z)\\
&=	\frac{1}{2}(\n_z\n_t u^v_{\b{z}}+\n_{\b{z}}\n_t u^v_z)-\frac{\rho^2}{4}u^v_{\b{z}}(u^v_t\o{u^v_{\b{z}}}-\o{u^v_t}u^v_z)\\
&\quad -\frac{\rho^2}{4}u^v_z(u^v_t\o{u^v_z}-\o{u^v_t}u^v_{\b{z}}).
\end{aligns}
Thus, by \eqref{2.16} and \eqref{2.17}, \eqref{2.15} can also be
written as
\begin{multline}\label{2.18}
\frac{\b{q}}{\lambda^2_0}\n^F_zu^v_z+\frac{q}{\lambda^2_0}\n^F_{\b{z}}u^v_{\b{z}}=\frac{1}{2}(\n_z\n_{\b{z}} u^v_{t}+\n_{\b{z}}\n_z u^v_t)\\-\frac{\rho^2}{4}u^v_{\b{z}}(u^v_t\o{u^v_{\b{z}}}-\o{u^v_t}u^v_z)
 -\frac{\rho^2}{4}u^v_z(u^v_t\o{u^v_z}-\o{u^v_t}u^v_{\b{z}}).
\end{multline}
By Stokes' theorem and using (\ref{2.18}), the first term in the RHS of \eqref{second1} reduces to
\begin{aligns}\label{2.19}
&\quad \text{Re}\int_{\Sigma}-\frac{4\b{q}}{\lambda_0^2}\rho^2(\n_t u^v_z\o{u^v_{\b{z}}}+u^v_z\o{\n_tu^v_{\b{z}}})idz\wedge d\b{z}\\
&=	\text{Re}\int_{\Sigma}\frac{4\b{q}}{\lambda_0^2}\rho^2( u^v_t\o{\n^F_{\b{z}}u^v_{\b{z}}}+\n^F_zu^v_z\o{u^v_{t}})idz\wedge d\b{z}\\
&=\text{Re}\int_{\Sigma}4\rho^2\o{u^v_t}(\frac{\b{q}}{\lambda^2_0}\n^F_zu^v_z+\frac{q}{\lambda^2_0}\n^F_{\b{z}}u^v_{\b{z}})idz\wedge d\b{z}\\
&=\text{Re}\int_{\Sigma}4\rho^2\o{u^v_t}\left(\frac{1}{2}(\n_z\n_t u^v_{\b{z}}+\n_{\b{z}}\n_t u^v_z)\right.\\
&\left.\quad-\frac{\rho^2}{4}u^v_{\b{z}}(u^v_t\o{u^v_{\b{z}}}-\o{u^v_t}u^v_z)
 -\frac{\rho^2}{4}u^v_z(u^v_t\o{u^v_z}-\o{u^v_t}u^v_{\b{z}})\right)idz\wedge d\b{z}\\
 &=-2\int_{\Sigma}\rho^2(|\n_t u^v_z|^2+|\n_tu^v_{\b{z}}|^2)idz\wedge d\b{z}\\
 &\quad -2\int_{\Sigma}\rho^4\left(-2\text{Re}((u^v_t)^2\o{u^v_z}\o{u^v_{\b{z}}})+|u^v_t|^2(|u^v_z|^2+|u^v_{\b{z}}|^2)\right)\frac{i}{2}dz\wedge d\b{z}.
\end{aligns}
Substituting \eqref{2.19} into \eqref{second1} gives
\begin{aligns}\label{second2}
   	\frac{1}{2}\frac{d^2 E(t)}{dt^2}|_{t=0} &=\int_{\Sigma}\rho^2(|u^v_z|^2+|u^v_{\b{z}}|^2)\frac{|q|^2}{\lambda_0^4}idz\wedge d\b{z}\\
   	&\quad+\text{Re}\int_{\Sigma}\rho^2\o{u^v_t}(\frac{\b{q}}{\lambda^2_0}\n^F_zu^v_z+\frac{q}{\lambda^2_0}\n^F_{\b{z}}u^v_{\b{z}})idz\wedge d\b{z}.
   \end{aligns}
For the pullback complexified tangent bundle $u^*T_{\mb{C}}S=u^*K_S^*\oplus\o{u^*K_S^*}$, it can be  equipped with the following pointwise Hermitian inner product
\begin{align}\label{HI1}
\left(f_1\frac{\p}{\p v}+f_2\frac{\p}{\p \b{v}}, e_1\frac{\p}{\p v}+e_2\frac{\p}{\p\b{v}}\right):=\rho^2(f_1\o{e_1}+f_2\o{e_2}),	
\end{align}
and the global inner product is defined by 
\begin{align}\label{HI2}
\langle\cdot,\cdot\rangle=\int_{\Sigma}(\cdot,\cdot)\lambda_0^2\frac{i}{2}dz\wedge d\b{z}.	
\end{align}
Denote 
\begin{align*}
\p u:=\frac{\p u}{\p z}\otimes dz=(u^v_z\frac{\p}{\p v}+\o{u^v_{\b{z}}}\frac{\p}{\p\b{v}})\otimes dz\in A^0(\Sigma,u^*T_{\mb{C}}S\otimes K_{\Sigma}), 	
\end{align*}
and define $\n^{u^*T_{\mb{C}}S\otimes K_\Sigma}_z\frac{\p u}{\p z}$ by
\begin{align*}
\n^{u^*T_{\mb{C}}S\otimes K_\Sigma}_{\p/\p z}\p u	=\n^{u^*T_{\mb{C}}S\otimes K_\Sigma}_{\p/\p z}(\frac{\p u}{\p z}\otimes dz)=: (\n^{u^*T_{\mb{C}}S\otimes K_\Sigma}_z\frac{\p u}{\p z})\otimes dz. 
\end{align*}
It can be  easily checked that $\frac{2\b{q}}{\lambda_0^4}\n^{u^*T_{\mb{C}}S\otimes K_\Sigma}_z\frac{\p u}{\p z}$ is a global section of $u^*T_{\mb{C}}S$, and 
\begin{aligns}\label{2.20}
&\quad \left\langle \text{Re}\left(\frac{2\b{q}}{\lambda_0^4}\n^{u^*T_{\mb{C}}S\otimes K_\Sigma}_z\frac{\p u}{\p z}\right),\frac{\p u}{\p t}\right\rangle =\text{Re}\left\langle \frac{2\b{q}}{\lambda_0^4}\n^{u^*T_{\mb{C}}S\otimes K_\Sigma}_z\frac{\p u}{\p z},\frac{\p u}{\p t}\right\rangle\\
&=\text{Re}\int_{\Sigma}(\frac{2\b{q}}{\lambda_0^4}\n^F_zu^v_z\o{u^v_t}\rho^2+\frac{2\b{q}}{\lambda_0^4}\o{\n^F_{\b{z}}u^v_{\b{z}}}u^v_t\rho^2)\lambda_0^2\frac{i}{2}dz\wedge d\b{z}\\
&=\text{Re}\int_{\Sigma}\rho^2\o{u^v_t}(\frac{\b{q}}{\lambda^2_0}\n^F_zu^v_z+\frac{q}{\lambda^2_0}\n^F_{\b{z}}u^v_{\b{z}})idz\wedge d\b{z},
\end{aligns}
which is exactly the second term in the RHS of (\ref{second2}). Here $\frac{\p u}{\p z}:=u_*(\frac{\p}{\p z})$ and $\frac{\p u}{\p t}:=u_*(\frac{\p}{\p t})$. Note that $\n$ is the natural induced connection on $u^*T_{\mb{C}}S$,  then  the Jacobi operator is given by 
\begin{align}\label{2.21}
\mc{J}:=-\frac{1}{\lambda_0^2}\n_z	\n_{\b{z}}	-\frac{1}{\lambda_0^2}R(\bullet,\frac{\p u}{\p z})\frac{\p u}{\p\b{z}}
\end{align}
which acts on the smooth section of $u^*T_{\mb{C}}S$, where the curvature operator $R=\n^2$ is given by $R(X,Y):=[\n_X,\n_Y]-\n_{[X,Y]}$. Then 
\begin{prop}\label{prop1}
$\mc{J}$ is real, semi-positive and symmetric. Moreover 
\begin{align*}
\mc{J}\left(\frac{\p u}{\p t}\right)=-\text{Re}\left(\frac{2\b{q}}{\lambda_0^4}\n^{u^*T_{\mb{C}}S\otimes K_\Sigma}_z\frac{\p u}{\p z}\right).
\end{align*}
\end{prop}
\begin{proof}
	$\mc{J}$ is real since 
	\begin{align*}
	\o{\mc{J}}&=	-\frac{1}{\lambda_0^2}\n_{\b{z}}	\n_{z}	-\frac{1}{\lambda_0^2}R(\bullet,\frac{\p u}{\p \b{z}})\frac{\p u}{\p z}\\
	&=-\frac{1}{\lambda_0^2}\n_{z}	\n_{\b{z}}-\frac{1}{\lambda_0^2}R(\frac{\p u}{\p\b{z}},\frac{\p u}{\p z})-\frac{1}{\lambda_0^2}R(\bullet,\frac{\p u}{\p \b{z}})\frac{\p u}{\p z}\\
	&=-\frac{1}{\lambda_0^2}\n_{z}	\n_{\b{z}}-\frac{1}{\lambda_0^2}R(\bullet,\frac{\p u}{\p z})\frac{\p u}{\p \b{z}}=\mc{J}.
	\end{align*}
For any smooth vector fields $X$ and $Y$ of $u^*T_{\mb{C}}S$,
\begin{align*}
\langle\mc{J}X,Y\rangle&=-\int_{\Sigma}	(\n_{\b{z}}	\n_{z}X+R(X,\frac{\p u}{\p \b{z}})\frac{\p u}{\p z},Y)idz\wedge d\b{z}\\
&=\int_{\Sigma}	(	\n_{z}X,\n_{z}Y)+R(X,\frac{\p u}{\p\b{z}},\frac{\p u}{\p z}, \b{Y})\frac{i}{2}dz\wedge d\b{z}\\
&=-\int_{\Sigma}(X, \n_{\b{z}}	\n_{z}Y+R(Y,\frac{\p u}{\p \b{z}})\frac{\p u}{\p z})\frac{i}{2}dz\wedge d\b{z}\\
&=\langle X,\mc{J}(Y)\rangle,
\end{align*}
which implies $\mc{J}$ is symmetric. Here we used the notation $$R(X,Y,Z,\b{W}):=-(R(X,Y)Z,W)$$ for any four vectors $X,Y,Z,W$ of $u^*T_{\mb{C}}S$. Moreover, if $X=Y=f_1\frac{\p}{\p v}+f_2\frac{\p}{\p\b{v}}$, then 
\begin{align*}
	\langle\mc{J}X,X\rangle=\langle\o{\mc{J}}X,X\rangle=\int_{\Sigma}	(	\n_{z}X,\n_{z}X)+R(X,\frac{\p u}{\p\b{z}},\frac{\p u}{\p z}, \b{X})\frac{i }{2}dz\wedge d\b{z}\geq 0,
\end{align*}
where the last inequality follows from 
\begin{align}\label{R positive}
	R(X,\frac{\p u}{\p\b{z}},\frac{\p u}{\p z}, \b{X})=\frac{\rho^4}{2}(|f_1|^2|u^v_z|^2+|f_2|^2|u^v_{\b{z}}|^2-2\text{Re}(\o{f_1}f_2u^v_zu^v_{\b{z}}))\geq 0.
\end{align}
So the operator $J$ is semi-positive. Lastly, since $\mc{J}=\o{\mc{J}}$, so
\begin{align*}
\mc{J}\left(\frac{\p u}{\p t}\right)&= \mc{J}(u^v_t\frac{\p}{\p v})+\mc{J}(\o{u^v_t}\frac{\p}{\p\b{v}})=2\Re\left(\mc{J}(u^v_t\frac{\p}{\p v})\right)\\
&=2\text{Re}\left(-\frac{1}{\lambda_0^2}(\n_z\n_{\b{z}}u^v_t+R^v_{vtz}u^v_{\b{z}})\frac{\p }{\p v}\right)\\
&=-2\text{Re}\left(\frac{1}{\lambda_0^2}\n_t\n_zu^v_{\b{z}}\frac{\p }{\p v}\right)
=-2\text{Re}\left(\frac{1}{\lambda_0^2}\frac{\p}{\p t}\n_zu^v_{\b{z}}\frac{\p }{\p v}\right)\\
&=-2\text{Re}\left(\frac{\b{q}}{\lambda_0^4}\n_z^F u^v_z\frac{\p}{\p v}+\frac{\b{q}}{\lambda_0^4}\o{\n^F_{\b{z}}u^v_{\b{z}}}\frac{\p}{\p\b{v}}\right)\\
&=-\text{Re}\left(\frac{2\b{q}}{\lambda_0^4}\n^{u^*T_{\mb{C}}S\otimes K_\Sigma}_z\frac{\p u}{\p z}\right),
\end{align*}
where the third equality holds since 
\begin{align*}
\Re \left(R(u^v_t\frac{\p}{\p v},\frac{\p u}{\p z})\frac{\p u}{\p\b{z}}	\right)&=\Re\left(R^v_{vv\b{v}}u^v_t\o{u^v_{\b{z}}}u^v_{\b{z}}\frac{\p}{\p v}+R^{\b{v}}_{\b{v}v\b{v}}u^v_t\o{u^v_{\b{z}}}\o{u^v_z}\frac{\p}{\p\b{v}}\right)\\
&=\Re\left(R^v_{vv\b{v}}u^v_t\o{u^v_{\b{z}}}u^v_{\b{z}}\frac{\p}{\p v}-R^{v}_{vv\b{v}}\o{u^v_t}u^v_{\b{z}}u^v_z\frac{\p}{\p v}\right)\\
&=\Re\left(R^v_{vtz}u^v_{\b{z}}\frac{\p}{\p v}\right),
\end{align*}
and the fifth equality holds since $\n_zu^v_{\b{z}}=0$ at $t=0$, the sixth equality follows from \eqref{2.15}.
 The proof is complete.
\end{proof}
Substituting \eqref{2.0}, \eqref{2.20} and \eqref{2.21} into \eqref{second2} obtain
\begin{thm}\label{thm1}
The second derivative of energy function is given by
\begin{align}\label{2.22}
\frac{1}{2}\frac{d^2 E(t)}{dt^2}|_{t=0}	=\int_{\Sigma}|du|^2\frac{|q|^2}{\lambda_0^4}d\mu_{g_0}-\left\langle \mc{J}\left(\frac{\p u}{\p t}\right),\frac{\p u}{\p t}\right\rangle,
\end{align}
	where  $\mc{J}(\frac{\p u}{\p t})=-\text{Re}(\frac{2\b{q}}{\lambda_0^4}\n_z^{u^*T_{\mb{C}}S\otimes K_\Sigma}\frac{\p u}{\p z})$.
\end{thm}
\begin{rem}
Note that in \cite[(8), (9)]{Toledo}, Toledo also obtained two expressions on the second derivative of energy function involving the first and second derivatives of complex structure $\frac{\p J}{\p t}$, $\frac{\p^2 J}{\p t^2}$.
\end{rem}
For the second term in the RHS of (\ref{2.22}), it is independent of the  solution  of  $\mc{J}(\bullet)=-\text{Re}(\frac{2\b{q}}{\lambda_0^4}\n_z^{u^*T_{\mb{C}}S\otimes K_\Sigma}\frac{\p u}{\p z})$. More precisely,
\begin{prop}
For any  solution $V\in A^0(\Sigma, u^*T_{\mb{C}}S)$ of the equation
\begin{align}\label{Jacobi equation}
	\mc{J}(V)=-\text{Re}\left(\frac{2\b{q}}{\lambda_0^4}\n_z^{u^*T_{\mb{C}}S\otimes K_\Sigma}\frac{\p u}{\p z}\right),
\end{align}
we have 
\begin{align*}
\left\langle\mc{J}(V), V\right\rangle=\left\langle \mc{J}\left(\frac{\p u}{\p t}\right),\frac{\p u}{\p t}\right\rangle.
\end{align*}
\end{prop}
\begin{proof}
	Let  $V_1, V_2$ be any two solutions of \eqref{Jacobi equation}, then $\mc{J}(V_1-V_2)=0$. By Proposition \ref{prop1}, $\mc{J}$ is  symmetric, so 
 \begin{align*}
 	\langle\mc{J}(V_1),V_1\rangle=\langle\mc{J}(V_2),V_1\rangle=\langle V_2,\mc{J}(V_1)\rangle=\langle V_2, \mc{J}(V_2)\rangle=\langle \mc{J}(V_2), V_2\rangle,
 \end{align*}
which completes the proof.
\end{proof}
Combining above Proposition and Theorem \ref{thm1} we have
\begin{cor}\label{cor1}
	The second derivative of energy function is given by
\begin{align}\label{2.23}
\frac{1}{2}\frac{d^2 E(t)}{dt^2}|_{t=0}	=\int_{\Sigma}|du|^2\frac{|q|^2}{\lambda_0^4}d\mu_{g_0}-\left\langle \mc{J}\left(V\right),V\right\rangle,
\end{align}
where $V\in A^0(\Sigma, u^*T_{\mb{C}}S)$ is a  solution of \eqref{Jacobi equation}.
\end{cor}

\section{Convexity of energy function}\label{sec3}

In this section,  we will prove the (strict) convexity of energy function at its critical point and discuss the uniqueness of its critical points.

Denote by $\n^{0,1}$ the $(0,1)$-part of the connection $\n$. We   define a Hermitian pointwise inner product on the line bundle  $\o{K_{\Sigma}}$ by $$( d\b{z},d\b{z}):=\frac{1}{\lambda_0^2}.$$  Combining with \eqref{HI1}, \eqref{HI2}, there is an induced inner product $\langle\cdot,\cdot\rangle$ on the space $A^{0,1}(\Sigma, u^*T_{\mb{C}}S)$. Denote by $(\n^{0,1})^*$ the adjoint operator of $\n^{0,1}$ with respect to $\langle\cdot,\cdot\rangle$. Then the action on the space $A^{0,1}(\Sigma, u^*T_{\mb{C}}S)$, $(\n^{0,1})^*$ is given by
\begin{align*}
(\n^{0,1})^*=-\frac{1}{\lambda_0^2}\n_z.	
\end{align*}
Denote the Hodge-Laplacian of $\n^{0,1}$  by 
\begin{align*}
\Delta^{0,1}:=\n^{0,1}(\n^{0,1})^*+(\n^{0,1})^*\n^{0,1}.	
\end{align*}
In terms of $\Delta^{0,1}$
the Jacobi operator $\mc{J}$
 on smooth sections of $u^*T_{\mb{C}}S$ is
\begin{align}\label{3.1}
\mc{J}=\Delta^{0,1}+\mc{R},
\end{align}
where $\mc{R}(\bullet):=-\frac{1}{\lambda_0^2}R(\bullet,\frac{\p u}{\p z})\frac{\p u}{\p\b{z}}$. By \eqref{R positive},  $\mc{R}$ is semi-positive. Denote by 
\begin{align*}
\mu=\frac{\b{q}}{\lambda_0^2}d\b{z}\otimes\frac{\p}{\p z} 	
\end{align*}
  the harmonic Beltrimi differential associated to the holomorphic quadratic form $qdz^2$. Then 
  \begin{align*}
  i_{\mu}du=\frac{\b{q}}{\lambda_0^2}d\b{z} i_{\frac{\p}{\p z}}(du)=\frac{\b{q}}{\lambda_0^2}d\b{z}\otimes\frac{\p u}{\p z},
  \end{align*}
and using \eqref{2.16}
\begin{align}\label{3.6}
(\n^{0,1})^*i_{\mu}du=-\frac{1}{\lambda_0^2}\n_z(\frac{\b{q}}{\lambda_0^2}\frac{\p u}{\p z})=-\frac{\b{q}}{\lambda_0^4}\n_z^{u^*T_{\mb{C}}S\otimes K_{\Sigma}}\frac{\p u}{\p z}.	
\end{align}
From \eqref{3.1} and noting $\mc{R}$ is semi-positive, $\text{Ker}\mc{J}\subset \text{Ker}\Delta^{0,1}$, so 
\begin{align}\label{3.66}(\n^{0,1})^*i_{\mu}du\in (\text{Ker}\Delta^{0,1})^{\perp}\subset (\text{Ker}\mc{J})^{\perp}.\end{align}
Since $\mc{J}=\Delta^{0,1}+\mc{R}$ is a symmetric elliptic partial differential operator, so there exists the Green operator $\mc{J}^{-1}$ satisfies 
$$V=P_{\text{Ker}\mc{J}}(V)+\mc{J}\circ \mc{J}^{-1}(V)$$
for any $V\in A^0(\Sigma, u^*T_{\mb{C}}S)$; see e.g. \cite[Page 450, Corollary]{Kodaira}. Here $P_{\text{Ker}\mc{J}}$ denotes the orthogonal projection to $\text{Ker}\mc{J}$. From \eqref{3.66}, 
there is a unique solution $\mc{J}^{-1}((\n^{0,1})^*i_{\mu}du)\in (\text{Ker}\mc{J})^{\perp}$ to the equation $\mc{J}(\bullet)=(\n^{0,1})^*i_{\mu}du$. Let
\begin{align*}
V=	\mc{J}^{-1}((\n^{0,1})^*i_{\mu}du)+\o{\mc{J}^{-1}((\n^{0,1})^*i_{\mu}du)}=\mc{J}^{-1}(\text{Re}(2(\n^{0,1})^*i_{\mu}du)).
\end{align*}
Then $V$ is a solution of (\ref{Jacobi equation}).
Using \eqref{3.6} we find
\begin{aligns}\label{3.5}
\langle\mc{J}(V),V\rangle &=4\langle \text{Re}((\n^{0,1})^*i_{\mu}du),\mc{J}^{-1}(\text{Re}((\n^{0,1})^*i_{\mu}du))\rangle\\
&=\langle (\n^{0,1})^*i_{\mu}du+\o{(\n^{0,1})^*i_{\mu}du},\mc{J}^{-1}((\n^{0,1})^*i_{\mu}du)+\o{\mc{J}^{-1}((\n^{0,1})^*i_{\mu}du)}\rangle\\
&=2\langle (\n^{0,1})^*i_{\mu}du,\mc{J}^{-1}((\n^{0,1})^*i_{\mu}du)\rangle\\
&\quad+ 2\text{Re}\langle (\n^{0,1})^*i_{\mu}du,\mc{J}^{-1}(\o{ (\n^{0,1})^*i_{\mu}du})\rangle.
\end{aligns}
On the other hand, 
\begin{align}\label{3.2}
	\langle (\n^{0,1})^*i_{\mu}du,\mc{J}^{-1}((\n^{0,1})^*i_{\mu}du)\rangle\leq \langle (\n^{0,1})^*i_{\mu}du,(\Delta^{0,1})^{-1}((\n^{0,1})^*i_{\mu}du)\rangle
\end{align}
 with equality  if and only if $\mc{R}(\mc{J}^{-1}(\n^{0,1})^*i_{\mu}du)=0$; see e.g. \cite[(2.48)]{KWZ}. The RHS of the above is
 \begin{aligns}\label{3.3}
 \langle (\n^{0,1})^*i_{\mu}du,(\Delta^{0,1})^{-1}((\n^{0,1})^*i_{\mu}du)\rangle &=\langle i_{\mu}du,\n^{0,1}(\Delta^{0,1})^{-1}((\n^{0,1})^*i_{\mu}du)\rangle	\\
 &=\langle i_{\mu}du,(\Delta^{0,1})^{-1}\n^{0,1}((\n^{0,1})^*i_{\mu}du)\rangle\\
 &=\langle i_{\mu}du,(\Delta^{0,1})^{-1}\Delta^{0,1}(i_{\mu}du)\rangle\\
 &=\langle i_{\mu}du,(\text{Id}-\mb{H})(i_{\mu}du)\rangle\\
 &=\|i_{\mu}du\|^2-\|\mb{H}(i_{\mu}du)\|^2,
 \end{aligns}
where $\mb{H}$ denotes the harmonic projection to the space $\text{Ker}\Delta^{0,1}$, the second equality follows from $(\n^{0,1})^2=0$ on $\Sigma$, the fourth equality holds by the identity $\text{Id}=\mb{H}+(\Delta^{0,1})^{-1}\Delta^{0,1}$. Note that
\begin{multline}\label{3.4}
\|i_{\mu}du\|^2=\|\frac{\b{q}}{\lambda_0^2}d\b{z}\otimes\frac{\p u}{\p z}\|^2\\=	\int_{\Sigma}\rho^2 \frac{|q|^2}{\lambda_0^6}(|u^v_z|^2+|u^v_{\b{z}}|^2)d\mu_{g_0}=\frac{1}{2}\int_{\Sigma}|du|^2\frac{|q|^2}{\lambda_0^4}d\mu_{g_0}.
\end{multline}
Substituting \eqref{3.2}, \eqref{3.3} and \eqref{3.4} into \eqref{3.5} follows
\begin{multline*}
	\langle\mc{J}(V),V\rangle\leq \int_{\Sigma}|du|^2\frac{|q|^2}{\lambda_0^4}d\mu_{g_0}-2\|\mb{H}(i_{\mu}du)\|^2\\+2\text{Re}\langle (\n^{0,1})^*i_{\mu}du,\mc{J}^{-1}(\o{ (\n^{0,1})^*i_{\mu}du})\rangle.
\end{multline*}
By Corollary \ref{cor1} we conclude
\begin{prop}\label{prop2}
The second derivative of energy function at $t=0$ satisfies
\begin{align*}
\frac{d^2 E(t)}{dt^2}|_{t=0}\geq 4\left(\|\mb{H}(i_{\mu}du)\|^2-\text{Re}\langle (\n^{0,1})^*i_{\mu}du,\mc{J}^{-1}(\o{ (\n^{0,1})^*i_{\mu}du})\rangle\right),	
\end{align*}
	with the equality if and only if $\mc{R}(\mc{J}^{-1}(\n^{0,1})^*i_{\mu}du)=0$.
\end{prop}
Now we begin to prove our main theorem, stated
as Theorem \ref{thm0.1} in  Introduction.
\begin{thm}\label{thm2}
If $t_0\in\mc{T}$ is a critical point of energy function, then the energy fucntion is convex at $t_0$. If moreover,  $du_{t_0}$ is never zero, then 	the energy function is strictly convex at $t_0$.
\end{thm}
\begin{proof}
If $t_0\in \mc{T}$ is a critical point of energy function, then the associated harmonic map $u$ at $t=0$ is $\pm$ holomorphic by Proposition \ref{prop0}. Without loss of generality, we may assume that $u$ is holomorphic. Denote 
\begin{align*}
	\Omega:=\{p\in \Sigma| u^v_z(p)\neq 0\},
\end{align*}
which is dense in $\Sigma$. In $\Omega$, we have 
\begin{aligns}\label{3.7}
(\n^{0,1})^*i_{\mu}du &=	-\frac{\b{q}}{\lambda_0^4}\n_z^{u^*T_{\mb{C}}S\otimes K_{\Sigma}}\frac{\p u}{\p z}\\
&=-\frac{\b{q}}{\lambda_0^4}\n^{u^*T_{\mb{C}}S\otimes K_{\Sigma}}_z u^v_z \frac{\p}{\p v}\\
&=-\frac{\b{q}}{\lambda_0^4}\frac{\n^{u^*T_{\mb{C}}S\otimes K_{\Sigma}}_z u^v_z}{u^v_z}\cdot u^v_z \frac{\p}{\p v}\\
&=-\frac{\b{q}}{\lambda_0^4}\p_z\log|du|^2\cdot u^v_z \frac{\p}{\p v},
\end{aligns}
where the first equality follows from (\ref{3.6}), the second equality holds since $u$ is holomorphic. We assume that 
\begin{align}\label{3.8}
\mc{J}^{-1}((\n^{0,1})^*i_{\mu}du)=X_1+X_2
\end{align}
 where $X_1$ is a smooth section of $u^*K_S^*$, while $X_2$ is a smooth section of $\o{u^*K_S^*}$. Since $u$ is holomorphic,  on $\Omega$, $$\frac{\p}{\p v}=(u^v_z)^{-1}\frac{\p u}{\p z}.$$ Thus 
\begin{align*}
X_1=X_1^z\frac{\p u}{\p z},\quad X_2=X_2^{\b{z}}\frac{\p u}{\p\b{z}}	
\end{align*}
on $\Omega$ for some local smooth functions $X_1^z, X_2^{\b{z}}$.
By the definition of $\mc{J}$ \eqref{2.21},  we get on $\Omega$
\begin{align*}
\mc{J}(X_1) &=-\frac{1}{\lambda_0^2}\n_z	\n_{\b{z}}(X_1^z\frac{\p u}{\p z})	-\frac{1}{\lambda_0^2}R(X_1^z\frac{\p u}{\p z},\frac{\p u}{\p z})\frac{\p u}{\p\b{z}}\\
&=-\frac{1}{\lambda_0^2}(\n^{ K_{\Sigma}^*}_z	\p_{\b{z}}X_1^z)\frac{\p u}{\p z}-\frac{1}{\lambda_0^2}\p_{\b{z}}X^z_1\n^{u^*K_S^*\otimes K_{\Sigma}}_z\frac{\p u}{\p z}\\
&=-\frac{1}{\lambda_0^2}(\n^{ K_{\Sigma}^*}_z	\p_{\b{z}}X_1^z)u^v_z\frac{\p}{\p v}-\frac{1}{\lambda_0^2}\p_{\b{z}}X^z_1\p_z\log|du|^2\cdot u^v_z \frac{\p}{\p v},\end{align*}
which is a smooth section of $u^*K_S^*$ . Similarly, $\mc{J}(X_2)$ is a section of  $\o{u^*K_S^*}$ on $\Omega$. Let $\mc{J}$ act on both sides of  \eqref{3.8} and noting that $(\n^{0,1})^*i_{\mu}du$ is a section of $u^*K_S$ on $\Omega$, we get
\begin{align}\label{3.9}
\mc{J}(X_2)=0,\quad \mc{J}(X_1)=(\n^{0,1})^*i_{\mu}du	
\end{align}
 on $\Omega$. Since $\o{\Omega}=\Sigma$, so \eqref{3.9} holds on $\Sigma$. Thus 
\begin{align}\label{3.10}
	\langle (\n^{0,1})^*i_{\mu}du,\mc{J}^{-1}(\o{ (\n^{0,1})^*i_{\mu}du})\rangle=\langle \mc{J}(X_1),\o{X_1}+\o{X_2}\rangle=\langle \mc{J}(X_1),\o{X_1}\rangle.
\end{align}
Since on $\Omega$, $\mc{J}(X_1)$ is a section of $u^*K_S^*$, while $\o{X_1}$ is a section of $\o{u^*K_S^*}$, then the pointwise inner product  
\begin{align*}
(\mc{J}(X_1), \o{X_1})\equiv0	
\end{align*}
on $\Omega$, so is on $\Sigma=\o{\Omega}$. Therefore
\begin{align}\label{3.11}
	\langle \mc{J}(X_1),\o{X_1}\rangle=0.
\end{align}
Furthermore $\mc{J}(X_2)=0$, and using $\langle \mc{J}(X_2), X_2\rangle=0$
 we get  $\mc{R}(X_2)=0$. Thus 
\begin{align}\label{3.12}
\mc{R}(\mc{J}^{-1}(\n^{0,1})^*i_{\mu}du)=\mc{R}(X_1+X_2)=\mc{R}(X_1)=-\frac{1}{\lambda_0^2}R(X_1,\frac{\p u}{\p z})\frac{\p u}{\p\b{z}}=0
\end{align}
 since $X_1$ is parallel to $\frac{\p u}{\p z}$ on $\Omega$. By using \eqref{3.10}, \eqref{3.11}, \eqref{3.12}
and Proposition \ref{prop2}, we have
\begin{align}\label{second3}
	\frac{d^2 E(t)}{dt^2}|_{t=0}= 4\|\mb{H}(i_{\mu}du)\|^2\geq 0,
\end{align}
i.e. the energy is convex at the critical point $t_0\in \mc{T}$. If moreover,  $du_{t_0}$ is never zero,  i.e.
at $t=0$, 
\begin{align*}
du=u^v_z dz\otimes \frac{\p}{\p v}+\o{u^v_z} d\b{z}\otimes \frac{\p}{\p \b{v}}
\end{align*}
 is never zero, which is equivalent to $u^v_z$ is also never zero, i.e. $\Omega=\Sigma$.  Then 
$$X_1=X_1^z\frac{\p u}{\p z}=X^z_1 u^v_z\frac{\p}{\p v}$$
on  $\Sigma$. Denote $Y:=X_1^z\frac{\p}{\p z}\in A^0(\Sigma, K_{\Sigma}^*)$, then 
\begin{align*}
X_1=du(Y).	
\end{align*}
Since $\mc{R}(X_1)=0$, so 
\begin{align*}
\Delta^{0,1}(X_1)=\mc{J}(X_1)=(\n^{0,1})^*(i_\mu du),	
\end{align*}
which implies that 
\begin{align*}
\n^{0,1}X_1-i_\mu du\in \text{Ker}(\n^{0,1})^*\cap 	\text{Ker}\n^{0,1}
\end{align*}
is harmonic, so 
\begin{align*}
\n^{0,1}X_1-i_\mu du=\mb{H}(\n^{0,1}X_1-i_\mu du)=-\mb{H}(i_\mu du).	
\end{align*}
Thus $\mb{H}(i_{\mu}du)=0$ if and only if $\n^{0,1}X_1=i_\mu du$, i.e.
\begin{align*}
	(\b{\p}X_1^z)u^v_z\frac{\p}{\p v}=\frac{\b{q}}{\lambda_0^2}d\b{z}u^v_z\frac{\p}{\p v}.
\end{align*}
Since $u^v_z$ is never zero, so $\b{\p}X_1^z=\frac{\b{q}}{\lambda_0^2}d\b{z}$,
 which implies that 
\begin{align*}
\b{\p}Y=	\frac{\b{q}}{\lambda^2_0}d\b{z}\otimes \frac{\p}{\p z}=\mu\in {H}^{0,1}(\Sigma, K_{\Sigma}^*),
\end{align*}
 which means the harmonic Betrimi differential $\mu$ is $\b{\p}$-exact, so $q\equiv 0$, which is a contradiction. Thus, the energy function is strictly convex at $t_0\in \mc{T}$. The proof is complete.
\end{proof}
\begin{rem}
	More generally one may consider the harmonic maps from Riemann surfaces to  a general Riemannian manifold $(N,g)$ with non-positive Hermitian  sectional curvature, i.e. 
	$$R(X,\b{X},Y,\b{Y})\leq 0$$
	for any two complex vectors $X,Y\in T_{\mb{C}}N:=TN\otimes \mb{C}$, $R=\n^2$ denotes the Riemannian curvature tensor, $\n$ is the Levi-Civita connection. Fix a non-trival homotpy class $[u_0]$ with $u_0:\Sigma\to N$,  for each complex structure $t\in \mc{T}$, if the associated harmonic map $u_t$ is unique and homotopic to $u_0$, then we also obtain an energy function $E(t)$  on the Teichm\"uller space $\mc{T}$ of $\Sigma$ as in \eqref{energy function}. In local coordinates, denote the Riemannian metric by  $g=g_{ij}dx^i\otimes dx^j$. Then the first derivative of the energy function is 
	\begin{align*}
\frac{dE(t)}{dt}|_{t=0}
=\Re\int_{\Sigma}g_{ij}u^i_z u^j_z(-\frac{4\b{q}}{\lambda_0^4})d\mu_{g_0}.
\end{align*}
Since $u$ is harmonic, so $(g_{ij}u^i_z u^j_z)dz^2$ is a holomorphic quadratic differential. $(g_{ij}u^i_z u^j_z)dz^2=0$ if and only if $t_0=\Gamma(0)\in\mc{T}$ is a critical point of the energy function.  Similarly, the second derivative of the energy function is given by 
\begin{align*}
	\frac{1}{2}\frac{d^2 E(t)}{dt^2}|_{t=0} &=\int_{\Sigma}|du|^2\frac{|q|^2}{\lambda_0^4}d\mu_{g_0}+\Re\int_{\Sigma}\frac{4\b{q}}{\lambda_0^4}	g_{ij}(\n^{u^*TN\otimes K_\Sigma}_z u^i_z) u^j_td\mu_{g_0}\\
	&=\int_{\Sigma}|du|^2\frac{|q|^2}{\lambda_0^4}d\mu_{g_0}-\left\langle \mc{J}\left(\frac{\p u}{\p t}\right),\frac{\p u}{\p t}\right\rangle,
\end{align*}
	where $\mc{J}$ is given by \eqref{2.21}  and 
	\begin{align*}
	\mc{J}\left(\frac{\p u}{\p t}\right)=\text{Re}\left(\frac{2\b{q}}{\lambda_0^4}\n_z^{u^*TN\otimes K_\Sigma}\frac{\p u}{\p z}\right).	
	\end{align*}
Here the inner product $\langle\cdot,\cdot\rangle$ is given by 
\begin{align*}
\langle X,Y\rangle:=\int_{\Sigma}2g_{ij}X^iY^j d\mu_{g_0}	
\end{align*}
for any two smooth sections $X=X^i\frac{\p}{\p x^i},Y=Y^i\frac{\p}{\p x^i}$ of $u^*TN$. We can also prove Proposition \ref{prop2} in this case as follows. Denote 
$$\Omega:=\left\{p\in \Sigma| \frac{\p u}{\p z}(p)\neq 0\right\}.$$
Since $[u_0]$ is non-trivial, so $du\not\equiv 0$, and also $\frac{\p u}{\p z}\not\equiv 0$. Thus $\Omega\neq \emptyset$. By harmonicity equation, 
 $\o{\Omega}=\Sigma$. In this case, if $t_0$ is a critical point, then $\frac{\p u}{\p z}\perp \frac{\p u}{\p\b{z}}$.
Since $|\frac{\p u}{\p z}|=|\frac{\p u}{\p \b{z}}|$, so $u$ is an immersion on $\Omega$. $u(\Sigma)$ is called a {\it totally geodesic immersed submanifold}  on $\Omega$ if $\n_{u_*(X)}u_*(Y)\in u_*(T_{\mb{C}}\Sigma)$ for any local smooth vector fields  $X,Y$ of $T_{\mb{C}}\Sigma=T\Sigma\otimes\mb{C}$ on $\Omega$, i.e. the second fundamental form of $u(\Sigma)$ vanishes on $\Omega$. Under the assumption of $u(\Sigma)$ is a totally geodesic immersion submanifold almost everywhere and $t_0=\Gamma(0)$ is a critical point, we also have
\begin{align*}
	(\n^{0,1})^*(i_\mu du)=-\frac{2\b{q}}{\lambda_0^4}\p_z\log |du|^2\frac{\p u}{\p z}
\end{align*}
almost everywhere, and thus
\begin{align*}
	\frac{d^2 E(t)}{dt^2}|_{t=0}= 4\|\mb{H}(i_{\mu}du)\|^2\geq 0.
\end{align*}
i.e., the energy function  is convex at the critical point $t_0\in\mc{T}$. Moreover, if $du_{t_0}$ is never zero, the same argument as the proof of Theorem  \ref{thm2}, the energy function
is strictly convex at $t_0\in \mc{T}$.
\end{rem}
As an application of Theorem \ref{thm2}, we have
\begin{cor}\label{cor2}
If $u_0: \Sigma\to S$ is a covering map, then there exists a unique complex structure $t_0\in \mc{T}$ such that the associated harmonic map $u_{t_0}$ is $\pm$ holomorphic, and  
\begin{align*}
E(t)\geq E(t_0)=\text{Area}(\Sigma). 	
\end{align*}
Moreover, the energy density satisfies $\frac{1}{2}|du|^2(t_0)\equiv 1.$  Indeed, the unique hyperbolic metric on $\Sigma$ which
minimizes the energy is the pull-back hyperbolic metric via $u_{t_0}$.
In this case, 
\begin{align*}
\frac{d^2 E(t)}{dt^2}|_{t=0}=4\|\mu\|^2_{\text{WP}}>0.	
\end{align*}

\end{cor}
\begin{proof}
If $u_0:\Sigma\to S$ is a covering map, then $u_0$ is a surjective map
with $\deg u_0\neq 0$, and for any $p\in \Sigma$, the induced
homomorphism $(u_0)_*: \pi_1(\Sigma, p)\to \pi_1(S, u_0(p))$ is
injective;  see e.g. \cite[Theorem 11.16]{Lee}. By \cite{SU1, SY}, the
energy function $E(t)$ is proper (see also \cite[Proposition 4.13]{DW}). Thus there exists a critical point $t_0\in \mc{T}$ such that 
$$E(t)\geq E(t_0)$$
for all $t\in\mc{T}$. From Proposition \ref{prop0}, the associated harmonic map $u_{t_0}$ is $\pm$ holomorphic. So
\begin{align*}
E(t_0)&=\int_{\Sigma}\rho^2(|u^v_z|^2+|u^v_{\b{z}}|^2)\frac{i}{2}dz\wedge d\b{z}\\
&=\left|\int_{\Sigma} u^*(\rho^2\frac{i}{2}dv\wedge d\b{v})\right|	
=\left|\deg u\int_S \rho^2\frac{i}{2}dv\wedge d\b{v}\right| \\
&=2\pi|\deg u_0||\chi(S)|
=2\pi|\chi(\Sigma)|\\
&=\text{Area}(\Sigma),
\end{align*}
where the first equality holds by \eqref{energy}, the second equality holds since $u$ is $\pm$ holomorphic, the third equality holds by the definition of degree, the fifth equality follows from the identity 
\begin{align}\label{3.13}
\chi(\Sigma)=\deg u_0\cdot\chi(S) 
\end{align} for covering maps; see e.g. \cite[Theorem 1.1]{KLP}, and the last equality follows from \eqref{Area}.
 
For any critical point $t\in \mc{T}$,
$u_{t}$ is $\pm$ holomorphic, and the Riemann-Hurwitz formula
gives
\begin{align*}
\chi(\Sigma)=\deg u_t\cdot\chi(S)-\sum(d_i-1)=\deg u_0\cdot\chi(S)-\sum(d_i-1),	
\end{align*}
where  $d_i\geq 1$ is the ramification index; see \cite[Theorem 1.2 or (1.1)]{KLP}. Combing with (\ref{3.13}) shows that $d_i=1$, so $\frac{\p u_t}{\p z}$ is never zero,  so is $du_{t}$.  By Theorem \ref{thm2}, we conclude the energy function is strictly convex at its critical points. Thus $E(t)$ is a Morse function with only index zero. From \cite[Lemma 4.11 (2)]{JE} and noting that $\mc{T}$ is connected, the energy function $E(t)$ has a unique critical point $t_0\in\mc{T}$.

Without loss of generality, we may assume $u_{t_0}$ is holomorphic, so
$$\frac{1}{2}|du|^2(t_0)=\frac{\rho^2}{\lambda_0^2}|u^v_z|^2.$$ 
Since $u^v_z$ is local holomorphic and with no zero point so
\begin{align*}
\frac{1}{\lambda_0^2}\p_z\p_{\b{z}}\log (\frac{1}{2}|du|^2)&=\frac{1}{\lambda_0^2}\p_z\p_{\b{z}}\log	(\frac{\rho^2}{\lambda_0^2}|u^v_z|^2)=\frac{1}{\lambda_0^2}\p_z\p_{\b{z}}\log	\frac{\rho^2}{\lambda_0^2}\\
&=\frac{\rho^2}{\lambda_0^2}|u^v_z|^2\cdot \frac{1}{\rho^2}\p_v\p_{\b{v}}\log \rho^2-\frac{1}{\lambda_0^2}\p_z\p_{\b{z}}\log\lambda_0^2\\
&=\frac{1}{2}\left(\frac{1}{2}|du|^2-1\right).
\end{align*}
See \cite[(16)--(19)]{SY0} for a general formula on the Laplacian of
energy density. By maximal principle, i.e., if $z_0$ is a maximum
point for the energy density, then the above is less than zero, hence
$\frac{1}{2}|du|^2-1\leq 0$ for all $z$ in $\Sigma$. But
$E(t_0)=\text{Area}(\Sigma)$ implies that,
$\frac{1}{2}|du|^2(t_0)\equiv 1$. By \eqref{3.7},
$(\n^{0,1})^*i_{\mu}du =0$, i.e. $i_{\mu}du$ is harmonic.
Thus 
\begin{align*}
	\|\mb{H}(i_{\mu}du)\|^2=\|i_{\mu}du\|^2=\frac{1}{2}\int_{\Sigma}|du|^2\frac{|q|^2}{\lambda_0^4}d\mu_{g_0}=\int_{\Sigma}\frac{|q|^2}{\lambda_0^4}d\mu_{g_0}=\|\mu\|^2_{\text{WP}},
\end{align*}
by the definition of Weil-Petersson metric \eqref{WP}.
From \eqref{second3}, we have 
\begin{align*}
\frac{d^2 E(t)}{dt^2}|_{t=0}=4\|\mu\|^2_{\text{WP}}>0.	
\end{align*}
Since the energy density is identically $1$, so the pull-back metric by $u_{t_0}$ is 
\begin{align*}
u^*_{t_0}(\rho^2dvd\bar{v})=\rho^2|u^v_z|^2dzd\bar{z}=\lambda_0^2 dzd\bar{z},	
\end{align*}
which is the unique hyperbolic metric minimizing the energy. The proof is complete.
\end{proof}
\begin{rem}
For any surjective map $u:\Sigma\to S$, by \cite[Lemma 3.6.2]{Jost}, the energy has a uniform lower bound. More precisely,
 $$E(u)\geq \text{Area}(S)$$ with equality if and only if $u$ is $\pm$ holomorphic and $|\deg u|=1$. In fact, 
\begin{align*}
E(u) &=\int_{\Sigma}\rho^2(|u^v_z|^2+|u^v_{\b{z}}|^2)\frac{i}{2}dz\wedge d\b{z}\\
&\geq \left|\int_{\Sigma}\rho^2(|u^v_z|^2-|u^v_{\b{z}}|^2)\frac{i}{2}dz\wedge d\b{z}\right|	\quad \text{with equality iff $u$ is $\pm$ holomorphic}\\
&=\left|\int_{\Sigma}u^*(\rho^2\frac{i}{2}dv\wedge d\b{v})\right|=|\deg u|\text{Area}(S)\\
&\geq \text{Area}(S)\quad \text{with equality iff $|\deg u|=1$}.
\end{align*}
\end{rem}

\begin{rem}
	If the fixed smooth map $u_0$ is identity map, $S=\Sigma$.
        Corollary \ref{cor2} was proved in \cite[Theorem
        3.1.3]{Tromba}. The energy density satisfies
        $\frac{1}{2}|du|^2\geq 1$ for any harmonic map $u$ homotopic
        to identity,  with the equality if and only if  $u$ is
        identity  \cite[Lemma 5.1]{Wolf}.
        Both of the above results were proved  essentially using Schoen-Yau \cite[Theorem 3.1]{SY0} or Sampson \cite[Proposition 1, Theorem 11]{Sampson} that these harmonic maps are orientation preserving diffeomorphisms.  
\end{rem}

\begin{rem}
Following \cite{Lab, Lab1, Lab2}, a Fuchsian representation of $\pi_{1}(\Sigma)$ in $\text{PSL}(n,\mb{R})$ is a representation which factors through the irreducible representation of $\text{PSL}(2,\mb{R})$ in $\text{PSL}(n,\mb{R})$ and a cocompact representation of $\pi_1(\Sigma)$ in $\text{PSL}(2,\mb{R})$. A Hitchin representation is a representation which may be deformed to a Fuchsian representation. The space of Hitchin representation is denoted by 
$$\text{Rep}_H(\pi_1(\Sigma),\text{PSL}(n,\mb{R}))$$
and is called a Hitchin component. In \cite{Hitchin}, 
Hitchin gives explicit parametrisations of Hitchin components. Namely, given a choice of a complex structure $J$ over a given compact surface $\Sigma$, he produces a homeomorphism 
$$H_J: \mc{Q}(2,J)\oplus\cdots\oplus \mc{Q}(n,J)\to
\text{Rep}_H(\pi_1(\Sigma),\text{PSL}(n,\mb{R})),$$ where
$\mc{Q}(p,J)$ denotes the space of holomorphic $p$-differentials on
Riemann surface $(\Sigma,J)$; see also \cite[Section 9.1]{Lab1} for
the construction of $H_J$. For each $J$ and $\rho$,  there exists a
unique (up to isometries) $\rho$-equivariant harmonic map
$f:\wt{\Sigma}\to \text{SL}(n,\mb{R})/\text{SO}(n)$, and thus obtain
an energy function $e_{\rho}(J)$ on Teichm\"uller space
(\cite{Corlette, Donaldson} or \cite[Proposition 5.5.2]{Lab2}).
In particular, the $\rho$-equivariant harmonic map $f$ is always an
immersion
\cite[Theorem 1.1]{Sanders}.
Note that $H_J$ is not mapping class group invariant, so in \cite[Section 9.2]{Lab1}, Labourie defines an equivariant Hitchin map by 
$$H: \mc{E}^{(n)}\to \text{Rep}_H(\pi_1(\Sigma),\text{PSL}(n,\mb{R})),\quad (J,\omega)\mapsto H_J(0,\omega),$$
where $\mc{E}^{(n)}$ is the vector bundle over Teichmu\"uller space
whose fibre above the (isotopy class of the) complex structure $J$
is $$\mc{E}^{(n)}_J:=\mc{Q}(3,J)\oplus\cdots\oplus \mc{Q}(n,J).$$ In
particular,  the Hitchin map is  surjective and the energy function
$e_{\rho}(J)$ is proper \cite{Lab2}. In \cite{Lab}  and \cite[Conjecture 9.2.3]{Lab1} Labourie
conjectured the Hitchin map is a homeomorphism, and this conjecture is
also equivalent to that the energy function $e_{\rho}(J)$ has a unique
critical point for any
Hitchin representation $\rho$.
This is also a motivation for the authors to study the (strict) convexity of energy function at critical points. Our method here seems to give a new proof on Labourie's conjecture for the case of Fuchsian representations.
 Also note that in \cite{Li}, for a fixed  complex structure, Q. Li proved that the energy density satifies $\frac{1}{2}|df|^2\geq 1$ for any Hitchin representation $\rho$ and the equality holds at one point only if $\frac{1}{2}|df|\equiv 1$  and $\rho$ is the Fuchsian representation.
\end{rem}


\begin{thebibliography}{99}


\bibitem{Ahlfors} L. Ahlfors, {\it Some remarks on Teichm\"uller's space of Riemann surfaces}, Ann. Math. {\bf 74} (1961), 171-191.



\bibitem{BC} P. Baird, J. Wood, Harmonic Morphisms Between Riemannian Manifolds, Clarendon Press, Oxford, 2003.

\bibitem{Bers} L. Bers, L. Ehrenpreis, {\it Holomorphic convexity of Teichm\"uller Spaces}, Bull. AMS
{\bf 70} (1964), 761-764.

\bibitem{JE} J. Cheeger, D. Ebin,  Comparison theorems in Riemannian geometry. North-Holland Mathematical Library, Vol. 9. North-Holland Publishing Co., Amsterdam-Oxford; American Elsevier Publishing Co., Inc., New York, 1975. viii+174 pp.

\bibitem{Chu} T. Chu, {\it The Weil-Petersson metric in moduli space}, Chinese J. Math. {\bf 4} (1976), 29-51.

\bibitem{Corlette} K. Corlette, {\it Flat G-bundles with canonical metrics}, J. Differential Geom. {\bf 28} (1988), no. 3, 361-382.

\bibitem{Donaldson} S.  Donaldson, {\it Twisted harmonic maps and the self-duality equations}, Proc. London Math. Soc. {\bf 55} (1987), 127-131.

\bibitem{DW} G. Daskalopoulos, R. Wentworth,  {\it Harmonic maps and Teichm\"uller theory}. Handbook of Teichm\"uller theory. Vol. I, 33-109, IRMA Lect. Math. Theor. Phys., 11, Eur. Math. Soc., Z\"urich, 2007.

\bibitem{ES} J. Eells, J. Sampson, {\it Harmonic mappings of Riemannian Manifolds}, American Journal of Mathematics {\bf 86} (1964), 109-160.

\bibitem{Hart} P. Hartmann, {\it On homotopic harmonic maps}, Can. J. Math. {\bf 19} (1967), 673-687.

\bibitem{Hitchin} N. Hitchin, {\it Lie Groups and Teichm\"uller spaces}, Topology {\bf 31} (1992), no. 3, 449-473.

\bibitem{Jost} J. Jost, Compact Riemann surfaces. An introduction to contemporary mathematics. Second edition. Springer-Verlag, Berlin, 2002. xvi+278 pp.

\bibitem{Jost1} J. Jost, S.-T. Yau, {\it Harmonic mappings and moduli spaces of Riemann surfaces}, Surveys in differential geometry. Vol. XIV. Geometry of Riemann surfaces and their moduli spaces, 171-196, Surv. Differ. Geom., 14, Int. Press, Somerville, MA, 2009.

\bibitem{Kodaira} K. Kodaira, Complex manifolds and deformations of complex structures, Grundlehren der Math. Wiss. 283, Springer (1986).

\bibitem{KWZ} I. Kim, X. Wan, G. Zhang, {\it Plurisubharmonicity and geodesic convexity of energy function
on Teichmu\"uller space}, arXiv:1809.00255.

\bibitem{KWZ1} I. Kim, X. Wan, G. Zhang, {\it Plurisuperharmonicity of reciprocal energy function on Teichm\"uller space and Weil-Petersson metrics},  arXiv:1901.05048.


\bibitem{Lab} F. Labourie, {\it Anosov flows, surface groups and curves in projective space}, Invent. Math. {\bf 165} (2006), 51-114.

\bibitem{Lab1} F. Labourie, {\it Flat projective structures on surfaces and cubic holomorphic differentials}, Pure. Appl. Math. Q., {\bf 3} (2007), no. 4, 1057-1099.

\bibitem{Lab2} F. Labourie, {\it Cross ratios, Anosov representations and the energy functional on Teichm\"uller space}, Ann. Sci. Ec. Norm. Sup\'er. {\bf 41} (2008), 1-38.
 

\bibitem{Lee} J. Lee, Introduction to topological manifolds. Second edition. Graduate Texts in Mathematics, 202. Springer, New York, 2011.

\bibitem{Li} Q. Li, {\it Harmonic maps for Hitchin representations}, Geom. Funct. Anal., (2019), https://doi.org/10.1007/s00039-019-00491-72019.


\bibitem{KLP} M. Kazaryan, S. Lando, V. Prasolov,  Algebraic curves. Towards moduli spaces. Translated from the 2018 Russian original by Natalia Tsilevich. Moscow Lectures, 2. Springer, Cham, 2018. xiv+231 pp.


\bibitem{Sanders} A. Sanders, {\it Hitchin harmonic maps are immersions}, 	arXiv:1407.4513, (2014), https://arxiv.org/pdf/1407.4513.pdf.

\bibitem{Sampson} J. Sampson, {\it Some properties and applications of harmonic mappings}, Annales Sci. Ecole Normale Superieure {\bf 11} (1978), 211-228.


\bibitem{SU} J. Sacks, K. Uhlenbeck,  {\it The existence of minimal immersions of two-spheres}, Bulletin of the American Mathematical Society, {\bf 83} (1977), no. 5, 1033-1036. 

\bibitem{SU1} J. Sacks, K. Uhlenbeck, {\it Minimal immersions of closed Riemann surfaces}, Trans. Amer. Math. Soc. {\bf 271} (1982), 639-652.


\bibitem{SU2} J. Sacks, K. Uhlenbeck, {\it The existence of minimal immersions of 2-spheres}, Ann. of Math. (2) {\bf 113} (1981), no. 1, 1-24.

\bibitem{SY0} R. Schoen, S.-T. Yau,  {\it On univalent harmonic maps between surfaces}, Inventiones Mathematicae {\bf 44} (1978), no. 3, 265-278.

\bibitem{SY} R. Schoen, S.-T. Yau, {\it Existence of incompressible minimal surfaces and the topology of 3-dimensional manifolds with non-negative sectional curvature}, Ann. Math. {\bf 110} (1979), 127-142.

\bibitem{Tromba1} A. Tromba, {\it On a natural algebraic affine connection on the space of almost complex structures and the curvature of Teichm\"uller space with respect to its Weil-Petersson metric}, Manuscripta Math. {\bf 56} (1986), 475-497. 


\bibitem{Tromba} A. Tromba,  Teichm\"uller theory in Riemannian geometry. Lecture notes prepared by Jochen Denzler. Lectures in Mathematics ETH Z\"urich. Birkh\"auser Verlag, Basel, 1992. 220 pp.

\bibitem{Toledo} D. Toledo, {\it Hermitian curvature and plurisubharmonicity of energy on Teichm\"uller space}, Geom. Funct. Anal. {\bf 22} (2012), no. 4, 1015-1032. 

\bibitem{Wolf0} M. Wolf, {\it The Teichm\"uller theory of harmonic maps}, J. Differential Geom. {\bf 29} (1989), no. 2, 449-479.

\bibitem{Wolf} M. Wolf,  {\it The Weil-Petersson Hessian of length on Teichm\"uller space}, J. Differential Geom. {\bf 91} (2012), no. 1, 129-169.

\bibitem{Wolpert1} S. Wolpert, {\it Noncompleteness of the Weil-Petersson metric for Teichm\"uller space}, Pacific J. Math {\bf 61} (1975), 573-577.

\bibitem{Wolpert2} S. Wolpert, {\it Chern forms and the Riemann tensor for the moduli space of curves}, Invent. Math. {\bf 85} (1986), 119-145.


\bibitem{Wolpert} S. Wolpert, {\it Geodesic length functions and the Nielsen problem,} J. Differential Geom. {\bf 25} (1987), no. 2, 275-296. 


\bibitem{Xin} Y. Xin, Geometry of harmonic maps, Progress in Nonlinear Differential Equations and their Applications, 23. Birkh\"auser Boston, Inc., Boston, MA, 1996. x+241 pp.

\bibitem{Yamada} S. Yamada, {\it Weil-Petersson  convexity of the energy function on classical and universal Teichm\"u{}ller spaces}, J. Differential Geom. {\bf 51} (1999), no. 1, 35-96.

\end{thebibliography}
\end{document}